\documentclass[pdflatex,sn-mathphys-num]{sn-jnl}

\usepackage{graphicx}%
\usepackage{multirow}%
\usepackage{amsmath,amssymb,amsfonts}%
\usepackage{amsthm}%
\usepackage{mathrsfs}%
\usepackage[title]{appendix}%
\usepackage{xcolor}%
\usepackage{textcomp}%
\usepackage{manyfoot}%
\usepackage{booktabs}%
\usepackage{algorithm}%
\usepackage{algorithmicx}%
\usepackage{algpseudocode}%
\usepackage{listings}%
\usepackage{bm}
\usepackage{tikz}
\usepackage{breqn}
\usepackage{subcaption}
\usepackage{amsmath}  
\usepackage{mathrsfs} 
\usepackage{bm}       
\usetikzlibrary{arrows.meta, positioning, calc}
\usepackage{caption}
\usepackage{stmaryrd}
\makeatletter
\newcommand{\fixed@sra}{$\vrule height 2\fontdimen22\textfont2 width 0pt\shortrightarrow$}
\newcommand{\shortarrow}[1]{%
  \mathrel{\text{\rotatebox[origin=c]{\numexpr#1*45}{\fixed@sra}}}
}

\DeclareMathOperator*{\argmin}{arg\,min}
\makeatother

\newtheorem{theorem}{Theorem}
\newtheorem{remark}{Remark}%

\newcommand{\juannote}[1]{\textit{\textcolor{green}{[JUAN]: #1}}}

\begin{document}

\title[Article Title]{Sensitivity analysis of fractional linear systems based on random walks with negligible memory usage}

\author*[1]{\fnm{Andr\'es} \sur{Centeno}}\email{andres.sanchez@tecnico.ulisboa.pt} 
\author[1,2]{\fnm{Juan} \sur{A. Acebr\'on}}\email{juan.acebron@ist.utl.pt}

\author[1]{\fnm{Jos\'e} \sur{Monteiro}}\email{jcm@inesc-id.pt}

\affil*[1]{\orgdiv{INESC-ID}, \orgname{Instituto Superior Técnico}, \orgaddress{\city{Universidade de Lisboa}, \country{Portugal}}}
\affil[2]{\orgdiv{Department of Mathematics}, \orgname{Carlos III University of Madrid}, \orgaddress{\country{Spain}}}

\abstract{A random walk-based  method is proposed to efficiently compute the solution of a large class of fractional in time linear systems of differential equations (linear F-ODE systems) with an inhomogeneous fractional exponent, along with the derivatives with respect to the system parameters. Such a method is unbiased and unconditionally stable, and can therefore be used to provide  an unbiased estimation of individual entries of the solution, or the full solution. By using stochastic differentiation techniques, it can be used as well to provide unbiased estimators of the sensitivities of the solution with respect to the problem parameters without any additional computational cost. The time complexity of the algorithm is discussed here, along with suitable variance bounds, which prove in practice the convergence of the algorithm. Finally, several test cases with random matrices were run to assess the validity of the algorithm, together with a more realistic example.}



\pacs[Acknowledgements]{
This work was supported by national funds through Fun\-da\-ção para a Ciência e a Tecnologia under projects URA-HPC PTDC/08838/2022 and   UIDB/50021/2020 (DOI:10.54499/UIDB/50021/2020).
JA was funded by  Ministerio de Universidades and specifically the requalification program of the Spanish University System 2021-2023 at the Carlos III University.}
\maketitle

\section{Introduction}\label{sec1}

In today's scientific landscape, the fast computation of sensitivities of a numerical method has become as important as finding a fast numerical method itself \cite{razaviFutureSensitivityAnalysis2021}. Sensitivity information is used typically when modeling natural phenomena and enginereering problems for determining the most relevant parameters influencing the behavior of the simulations, and becomes crucial many times in the design of the experiments, data assimilation, and the reduction of complex models. Moreover, it is almost never the case that all parameters of a model are known, so users need to either complete the parameter vector by providing rough estimations or by learning the parameters from a few observations of the system \cite{hasanovhasanogluIntroductionInverseProblems2017}. The latter, when using gradient-based optimization, requires the scalable computation of sensitivities of the solution in order to find a local minima for the parameters in a reasonable time, as seen in the end loop of Fig. \ref{fig:flowchart}.

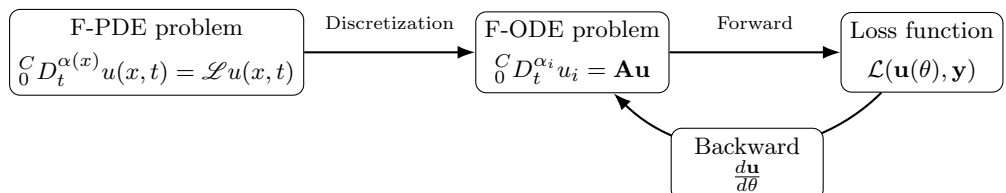
\begin{figure}[b]
    \centering
\begin{tikzpicture}[
    node distance=2.25cm,
    every node/.style={draw, align=center, rounded corners, minimum height=0.75cm, minimum width=2cm, font=\small},
    arrow/.style={-Latex, thick},
    ]

    \node (fpde) {F-PDE problem\\[1ex]$^C_0 D^{\alpha(x)}_t u(x,t) = \mathscr{L} u(x,t)$};
    \node (fode) [right=of fpde] {F-ODE problem\\[1ex]$^C_0 D^{\alpha_i}_t u_i = \mathbf{A}\mathbf{u}$};
    \node (loss) [right=of fode] {Loss function\\[1ex]$\mathcal{L}(\mathbf{u}(\theta),\mathbf{y})$};

    \draw[arrow] (fpde) -- (fode) node[midway, above, sloped, draw=none, font=\footnotesize] {Discretization};
    \draw[arrow] (fode) -- (loss) node[midway, above, sloped, draw=none, font=\footnotesize] {Forward};
    \draw[arrow, bend left=45] (loss) to (fode);

    \node at ($(loss)!0.5!(fode)$) [below=1cm, fill=white] {Backward\\$\frac{d\mathbf{u}}{d\theta}$};
\end{tikzpicture}\textbf{}    
\caption{Flowchart for fractional inverse problems. The F-PDE is discretized onto an F-ODE, then a loss is computed and parameters are updated to reduce that loss until it falls bellow a threshold or a maximum number of iterations is reached. Looping through the last step fast enough is critical for the algorithm to converge in a reasonable time.}
    \label{fig:flowchart}
\end{figure}
There has been an increasing push towards fractional order modeling in the last century \cite{west2014colloquium}. The main motivation was to explain experimental data characterized by a nonlinear mean squared displacement with respect to time, that traditional diffusion models cannot explain. Fractional derivatives account for that anomalous behaviour~\cite{zaburdaev2015levy}, and contain classical derivatives as a limiting case when the fractional exponent approaches an integer, which makes fractional modelling somehow a superset of classical modelling. Even to this date, fractional modelling is an active area of research finding experimental sub/superdiffusive behavior in phenomena long thought to be purely classical \cite{liu2023experimental}.

The non-locality of the fractional operators leads to more unstable and expensive numerical methods, both in terms of execution time and memory. As an example, the fractional Laplacian operator when discretized results in a dense matrix \cite{ilic2006numerical}. This leads not only to a heavy memory overhead but also to a lack of scalability for the corresponding algorithms. Moreover, the process of raising a matrix to a fractional power becomes highly non-trivial as well \cite{cusimano2018discretizations}.
In fractional in time order systems of differential equations (F-ODE systems), due to the loss of locality \cite{yan2018analysis} it is typically required for the numerical schemes to store the solution at all previous timesteps to advance forward in time. Many modern numerical methods explore alternatives to alleviate this issue \cite{diethelm2019fundamental}, but cannot ultimately avoid the complexity of the problem itself, which greatly magnifies when trying to repeatedly compute the sensitivities of the method for instance when solving an inverse problem.

In the inverse problem setting there is a clear distinction between linear and non-linear problems. Nowadays, the most current promising techniques for the non-linear case come from leveraging fractional models with deep learning techniques. The two most notable techniques, which adopt quite different approaches, are neural fractional ordinary differential equations \cite{jafarian2017artificial,antil2020fractional} and fractional Physics-Informed Neural Networks (fPINNs) \cite{pang2019fpinns,guo2022monte}. The latter is state-of-the-art for inverse problems for non-linear fractional partial differential equations (F-PDEs). In our work we are concerned only with the fractional-in-time linear case, for which there is long-established research on fast and stable numerical algorithms \cite{diethelm2004numerical,magin2011fractional}, and methods for solving the corresponding inverse problems \cite{bondarenko2009numerical,miller2013coefficient,jiang2018bayesian}. 

However, scalable methods to efficiently compute the sensitivities of a F-ODE system are still lacking, and the aim of this paper is precisely moving forward in this direction by proposing a stochastic method for this purpose. Stochastic methods were not proposed to compete in the same problems as deterministic methods. For small-sized problems, deterministic methods are the best approach. But for large scale problems, stochastic methods are often the only alternative, which is basically due to their embarrassingly parallel nature. In the continuum, almost all non-linear parabolic solvers in $d\geq 100$ Euclidean dimensions are based as a rule on some stochastic representation of the solution \cite{richter2021solving}. This has also been the trend in fractional parabolic problems. There has been a growing interest in Monte Carlo F-PDE solvers \cite{gulian2018stochastic} and parameter estimation techniques of fractional diffusion models \cite{kubiliusParameterEstimationFractional2017} through stochastic calculus techniques that go well beyond the standard Itô calculus. We also acknowledge that in recent years deep learning techniques have also exhibited remarkable results in high dimensions (we refer back to \cite{pang2019fpinns}), mainly because of the ability of neural networks to avoid a blow-up in the number of parameters with respect to Euclidean dimensions.

Nevertheless, the scalable sensitivity computation theme in discrete ODE/F-ODE systems remains largely unexplored, despite its importance. One example is graph Laplacian learning through diffusion ODE models \cite{thanou2017learning}. The regularized optimization algorithm proposed in that paper shows excellent F-measure results in reconstructing a graph from observations in the nodes, but this is merely for $n=20$ nodes in synthetic graphs and $n\leq 30$ nodes in real-world graphs. The computational complexity of the algorithm, being of order of $\mathcal{O}(n^3)$, makes it practically unfeasible to be used in experiments consisting in typical real-life social networks of sizes beyond $n\gg 10^6$ nodes \cite{10.1145/2567948.2576939}. In the fractional setting, as we mentioned above, scalability issues are even magnified, and still research on this topic is scarce. In fact, to the best of our knowledge, no algorithm has been proposed yet capable of solving the graph learning problem proposed in \cite{thanou2017learning} when characterized by an inhomogeneous subdiffusive model.

For this purpose we propose here a random walk-based algorithm capable of computing in parallel one entry of the solution of ODE/F-ODE systems and all their respective parameter sensitivities with a negligible memory cost per individual simulation with respect to time. Note that this is specially useful in inverse problems, since in a realistic experimental setting, the measurements are typically done at just a few nodes of the mesh. To cite a few examples, see current experiments in anomalous transport in porous media \cite{cortis2005computing, maryshev2013adjoint}, and heat transfer \cite{martin2015bayesian}, which are both described in the anomalous case by linear F-PDEs, as well as the computation of greeks in financial mathematics \cite{giles2005smoking}, e.g.  This has been accomplished by embedding the fractional nature of the problem directly onto the simulations, now being non-Markovian. Our work is based on previous work on stochastic matrix function methods \cite{acebron2019monte,acebron2020highly,guidotti2024stochastic,guidotti2024fast} that compute the solution at single nodes by using a recursive integral representation of the solution. Sensitivities are obtained simultaneously with the solution through either Malliavin weights or smooth perturbation analysis. This method is unbiased, meaning that, assuming perfect computation of the scalar Mittag-Leffler function and its scalar derivatives, the expected value of the algorithm is exactly the true solution.

The layout of this work is as follows. In Section \ref{background} we review the existing knowledge on stochastic methods for ODEs. Here we also discuss the failure of some classical sensitivity methods for PDEs when dealing with the discrete case, and describe some alternative methods that can be used to overcome this failure. In Section \ref{frac-lin-sys} we formulate the class of F-ODE systems we solve, and describe the inverse problems we intend to tackle.
Section \ref{stochastic-passes} describes the stochastic representation for both the solution and their derivatives, and provides an analysis of the variance of the methods and their time complexity. In Section \ref{numeric-results}, we numerically assess the validity of our algorithm through hypothesis testing for some generated random problems, and close showing the results corresponding to a test case application, which are in agreement with those obtained using a deterministic method. Finally, in Section \ref{conclusions}, we present our conclusions and outline the next steps to be taken.

\section{Background}
\label{background}

\subsection{Stochastic methods for computing the solution of differential equations}
As established in the '50s through the celebrated Feynman-Kac formula \cite{Karatzas1998}, there exists a strong connection between the theory of stochastic process (in particular, the diffusion process) and PDEs, offering a probabilistic interpretation to the solution of these equations. The function  $u(\mathbf{x},t)$, with ${\bf x} \in {\mathbb{R}}^n$, satisfying the Cauchy problem for the linear parabolic partial differential equation,
\begin{equation}
    \label{eq:Feynman-Kac}
    \frac{\partial u}{\partial t}=\mathcal{L}u-c(\mathbf{x},t)u,\qquad u(\mathbf{x},0)=f(\mathbf{x})
\end{equation}
can be represented probabilistically as
\begin{equation}
    \label{eq:Feynman-Kac_sol}
    u(\mathbf{x},t)=\mathbb{E}\left(f(\mathbf{X}_t)\exp\left(-\int_0^t d\tau\; c(\mathbf{X}_s,t-\tau) \right)\mid \mathbf{X}_0=\mathbf{x}\right).
\end{equation}
Here $\mathcal{L}$ is a linear elliptic operator, and $c({\bf x},t) \geq 0$  and $f({\bf x})$ represent the continuous bounded coefficients and initial condition, respectively. Concerning $\mathbf{X}_t$, this corresponds to the $n$-dimensional stochastic 
process starting at $\mathbf{X}_0=\mathbf{x}$ and the expected values are taken with respect to the corresponding measure. Such a stochastic process is in general the solution of the system of  
stochastic differential equations (SDEs) of the Itô type, related to the 
elliptic operator, 
\begin{equation}
        d {\bf X}(t) = {\bf b}(t,{\bf x}) \, dt + {\bm \sigma}(t,{\bf x}) \, d {\bf W}(t).
                                                         \label{SDE}
\end{equation}  
Here ${\bf W}(t)$ represents the $n$-dimensional standard Wiener process, ${\bf b}$ the drift , and ${\bm \sigma}$ the diffusion matrix, which are related to the  coefficients of the elliptic operator $\mathcal{L}$.

Although much less known, there also exists a probabilistic representation for the case of first-order ODEs. In fact, it was proved in \cite{acebron2020highly} that the solution of the initial value problem for the system of ODEs
\begin{equation}
    \frac{d\mathbf{u}}{dt}=\mathbf{Q}\mathbf{u},\quad \mathbf{u}(0)=\mathbf{u}_0,
\end{equation}
can be represented probabilistically as follows
\begin{equation}
    u_i(t)=\mathbb{E}\left(u_{X_t}(0) \mid X_0=i\right).
\end{equation}
Here $\mathbf{u}=(u_i)_{i=1,\ldots,n}$ and $\mathbf{u}(0)=(u_i(0))_{i=1,\ldots,n}$ are $n$-dimensional vectors, and $\{X_t\colon t\ge 0\}$ a stochastic process with finite state space $\Omega=\{1,2,\dots,n\}$ corresponding to a continuous-time Markov chain (CTMC) generated by the infinitesimal generator $\mathbf{Q}=(q_{ij})_{i,j=1,\ldots,n}$.
This result can be further generalized to deal with another class of ODEs through the following theorem:
\begin{theorem}
    A stochastic representation of the system of ODEs 
\begin{equation}
    \frac{d\mathbf{u}}{dt}=\mathbf{Q}\mathbf{u}+\mathbf{D}\mathbf{u},\quad \mathbf{u}(0)=\mathbf{u}_0,
\end{equation}
with $\mathbf{D}=(d_{ij})_{i,j=1,\ldots,n}$ an arbitrary diagonal matrix, is given by
\begin{equation}
    \label{eq:discrete-feynman-kac}
    u_i(t)=\mathbb{E}\left(u_{X_t}(0)\exp\left(\int_0^t d\tau\;d_{X_\tau,X_\tau}\right) \mid X_0=i\right)
\end{equation}
\end{theorem}
\begin{proof} Here we describe a sketch of the proof which is mostly based in \cite{acebron2020highly}. Using the Strang splitting for the matrix exponential 
\begin{equation}
        e^{(\mathbf{D}/2+\mathbf{Q}+\mathbf{D}/2)\Delta t}=e^{\mathbf{D}/2\Delta t}e^{\mathbf{Q}\Delta t}e^{\mathbf{D}/2\Delta t}+\mathcal{O}(\Delta t^3),
\end{equation} 
the time interval $[0,t]$ is divided in $N_t$ equispaced subintervals of size $\Delta t$, with $\Delta t=t/N_t$, $t_i=i\Delta t, i=0,\ldots,N_t$  and then, the solution of the system, $\mathbf{u}(t)=\exp\{t(\mathbf{D}+\mathbf{Q})\}\mathbf{u}(0)$, is approximated as follows
    \begin{equation}
        \mathbf{u}(t)=\left(\prod_{i=1}^{N_t} e^{\mathbf{D}/2\Delta t}e^{\mathbf{Q}\Delta t}e^{\mathbf{D}/2\Delta t}\right)\mathbf{u}(0)+\mathcal{O}(\Delta t^2).
    \end{equation}
    Note that in this approximation the action of $e^{\mathbf{Q}\Delta t}$ over a vector induces the time-evolution of the Markov chain, and when taking the limit $\Delta t \rightarrow 0$ this can be seen as the integration of $e^{\mathbf{D}/2\Delta t}$ through the CTMC path. In this limit the approximation converges to the solution and we recover therefore Equation (\ref{eq:discrete-feynman-kac}) due to the Markovianity of the stochastic process.
\end{proof}
Note that this stochastic solution for a system of ODEs resembles the Feynman-Kac formula for PDEs in Equation (\ref{eq:Feynman-Kac_sol}), where the diagonal matrix $\mathbf{D}$ plays the role of the coefficient $c(\mathbf{x},t)$ in the discrete setting.

These results can be generalized even further for a more general class of system of first-order ODEs such as
\begin{equation}
    \frac{d\mathbf{u}}{dt}=\mathbf{A}\mathbf{u},\quad \mathbf{u}(0)=\mathbf{u}_0, 
\end{equation}
where $\mathbf{A}=(a_{ij})_{i,j=1,\ldots,n}$ is a matrix with $a_{ii}<0$.
The stochastic representation becomes now more involved, since it is now required to take the account the sign of each matrix entries $a_{ij}$. This can be seen readily from the integral representation of the solution, which is given by
\begin{align}
    &u_i(t)=e^{a_{ii}t}u_i(0)\nonumber\\
    &+\sum_{k\neq i}\frac{\mathrm{sgn}(a_{ik})a_{ik}}{\sum_{j\neq i}\mathrm{sgn}(a_{ij})a_{ij}}\int_0^t d\tau\;\frac{\mathrm{sgn}(a_{ik})\sum_{j\neq i}\mathrm{sgn}(a_{ij})a_{ij}}{a_{ii}}(a_{ii}e^{a_{ii}\tau})u_k(t-\tau).
\end{align}
Therefore, the solution can be represented probabilistically as the expected value of a multiplicative functional over a path of a CTMC, which we denote hereafter as $\Gamma_t$, with the infinitesimal generator now given by
\begin{equation}
    q_{ij} = \begin{cases} 
        a_{ii}, & \text{if } i = j \\ 
        \dfrac{\mathrm{sgn}(a_{ik}) a_{ik}}{\sum_{j \neq i} \mathrm{sgn}(a_{ij}) a_{ij}}, & \text{if } i \neq j 
    \end{cases}
\end{equation}

\subsection{Stochastic methods for the sensitivity analysis of differential equations}

Differentiating numerical methods based on solving stochastic differential equations (SDEs) with respect to certain parameters is typically straightforward, since the state-space $\Omega\subset \mathbb{R}^n$ of the stochastic process is continuous, and small changes in the parameters often lead to small changes in the output.  But not any method of approximating the derivatives work. In fact,
assume $\theta\in \mathbb{R}$ a given parameter, and $V_{\theta}(\mathbf{x},t)$ the stochastic solution of the Equation (\ref{eq:Feynman-Kac}) with zero coefficient $c$ obtained through the Feynman-Kac formula as 
\begin{equation}
    V_{\theta}(\mathbf{x},t)=\mathbb{E}(f(\mathbf{X}(\theta)_{N_t})\vert \mathbf{X}(\theta)_{0} = \mathbf{x}) 
\end{equation}
where $\mathbf{X}(\theta)_{N_t}$ corresponds to the numerical solution of the SDEs obtained using an explicit Euler-Mayurama scheme with time-step $\Delta t$ \cite{gobet2016monte}, which is given by
\begin{equation}
    \label{eq:euler-mayurama}
    \mathbf{X}(\theta)_{i+1}=\mathbf{b}_{\theta}(t_i,\mathbf{X}(\theta)_{i})h+\bm{\sigma}_{\theta}(t_i,\mathbf{X}(\theta)_{i})\Delta{\bf W}_i\sqrt{h}, \quad i=0,\ldots,N_t.
\end{equation}
Here $\Delta {\bf W}_i={\bf W}_{i+1}-{\bf W}_i$ are independent and identically distributed unit normal random variables $\mathcal{N}(0,\mathbf{I})$.
To evaluate $\partial_\theta V_\theta(\mathbf{x},t)$, the often portrayed as naïve solution is to use as an estimator of the derivative a finite-difference scheme, which entails in practice running two different simulations for two different parameter values one perturbed by $\Delta \theta$ as follows
\begin{equation}
    \partial_\theta V_{\theta}(\mathbf{x},t)\approx\frac{1}{\Delta \theta}\left(\mathbb{E}(g(\mathbf{X}(\theta+\Delta \theta)_{N_t})\vert \mathbf{X}(\theta)_{0} = \mathbf{x})-\mathbb{E}(g(\mathbf{X}(\theta)_{N_t})\vert \mathbf{X}(\theta)_{0} = \mathbf{x})\right)
\end{equation}

However, this estimator suffers from the well-known bias-variance tradeoff, being the variance of order of $\mathcal{O}((\Delta \theta)^{-2})$ \cite{l1992convergence}. A solution is to use the so-called {\it reparametrization trick} \cite{mohamed2020monte}. The idea of the  reparametrization trick is to couple both simulations using the same random points $\omega$, and then take the limit $\Delta \theta \rightarrow 0$~\cite{mohamed2020monte} once the simulations are completed. For the solution of the SDE, this corresponds to using the same Gaussian increments for both simulations, one with parameter $\theta$, and the other one with $\theta+\Delta \theta$. For the first case, the solution is  $\mathbf{X}(\theta;\omega)_{N_t}$, while for the second case $\mathbf{X}(\theta+\Delta \theta;\omega)_{N_t}$. Under mild regularity assumptions it holds that $\Vert \mathbf{X}(\theta+\Delta \theta;\omega)_{N_t}-\mathbf{X}(\theta;\omega)_{N_t}\Vert=\mathcal{O}(\Delta \theta)$. Assuming also some regularity for $\mathbf{b}_{\theta}$ and $\bm{\sigma}_{\theta}$, the quantity $\partial_{\theta}\mathbf{X}(\theta;\omega)_{N_t}$ exists for a fixed $\omega$, and thus, making use of the dominated convergence theorem and the chain rule, we obtain an estimator for the sensitivities of
 $\partial_{\theta}V_{\theta}(x,t)$ \cite{giles2005smoking}:
\begin{equation}
    \partial_{\theta} V_{\theta}(\mathbf{x},t)=\mathbb{E}(\partial_{\theta} g(\mathbf{X}(\theta)_{N_t})\vert \mathbf{X}(\theta)_{0} = \mathbf{x}) 
\end{equation}
However, swapping both limits---the expected value, which is an integral, and the derivative
defined as limits which do not necessarily commute---does not hold for the ODE case, since the state-space for the stochastic representation of the solution is now discrete. We cannot fix $\omega$ and differentiate as it was done before, mainly because the estimator is not continuous with respect to $\theta$ anymore. Therefore, the reparametrization trick fails for this discrete setting. What follows is a simple example to illustrate the failure.

Let $X_\theta$ be an exponentially distributed random variable $X_\theta\sim p_\theta(x)=\theta e^{-\theta x}$. This corresponds indeed to the nature of the random variable used in generating the CTMC path. Let's suppose we want to compute the sensitivity of $\mathbb{E}(X_\theta)$ using the reparametrization trick. As we would do in an implementation, we separate the source of randomness from the parameter $\theta$ as follows.
\begin{equation}
    \label{reparam-continuous}
    X_\theta= -\frac{1}{\theta}\log(U),\qquad U\sim\mathcal{U}(0,1)
\end{equation} then, as the sequence of quotients that yield the derivative is bounded by an integrable function $\int_0^1 du\;\vert \log(u)\vert =1$ for small enough $\Delta \theta$
\begin{equation}
    f_{\Delta \theta}(u)=\frac{-\log(u)}{\theta+\Delta \theta}-\frac{-\log(u)}{\theta}=\frac{\log(u)\Delta \theta}{\theta(\theta+\Delta \theta)}\leq \log(u)
\end{equation}
we can apply the dominated convergence theorem \cite{mohamed2020monte} and swap the limits
\begin{equation}
    \partial_\theta \mathbb{E}(X_\theta)=\int_0^1 du\;\partial_\theta \frac{-\log(u)}{\theta}=\int_0^1 du\; \frac{\log(u)}{\theta^2}
\end{equation}
thus obtainig an estimator for both the expected value and the sensitivity using only single random number stream $U\sim \mathcal{U}(0,1)$.
\begin{equation}
\mathbb{E}(X_\theta)=\mathbb{E}_U\left(\frac{-\log(U)}{\theta}\right),\quad \partial_\theta \mathbb{E}(X_\theta)= \mathbb{E}_U\left(\frac{\log(U)}{\theta^2}\right)=-\frac{1}{\theta}\mathbb{E}(X_\theta)
\end{equation}
Now consider a new random variable $Y_\theta$ that depends on $X_\theta$ through an indicator function.
This corresponds to the simplest form of a two state $\bm{S}=\{0,1\}$ CTMC observed at time $T=\sqrt{2}$, starting at state $1$.
\begin{equation}
\label{reparam-discrete}
Y_\theta = \mathbf{1}_{(\sqrt{2},\infty)}(X_\theta)=\mathbf{1}_{(\sqrt{2},\infty)}\left(\frac{-\log(U)}{\theta}\right) 
\end{equation}
The reparametrization trick in this case clearly yields a non-differentiable, piecewise constant function of the parameter $\theta$ for a fixed realization $U$, even though the expected value of the random variable is actually a differentiable function of the parameter $\theta$ given by $\mathbb{E}(Y_\theta)=\mathbb{P}(X_\theta>\sqrt{2})=e^{-\theta \sqrt{2}}$.

This is indeed the case of the stochastic solution for the ODEs. Fortunately, many strategies have been proposed by the discrete event systems (DES) community for dealing with far more complex stochastic processes than those found here.  The main two pathways that we use are conditional Monte Carlo (CMC) \cite{fu_conditional_1997}---also known as smooth perturbation analysis---and the score function/likelihood/Malliavin weights method \cite{glynn1995likelihood}, having also considered measured-valued and stochastic automatic differentiation \cite{heidergott2010gradient,arya2022automatic}.
In this work, we use the Malliavin weights technique as it has been proven effective in SDEs \cite{warren2013malliavin} and non-Markovian random walks \cite{ertel2022operationally}. Then, for the time sensitivity, we opt for CMC as it has been already successfully applied to performance measures of generalized semi-Markov processes (GSMP) in Chapter 3 of~\cite{fu_conditional_1997}.



\section{Mathematical preliminaries}
\label{frac-lin-sys}
We consider the following system of fractional differential equations. In our notation, an $i$-th subscript of a vector is the projection onto the $i$-th canonical basis vector $e_i$ in $\mathbb{R}^n$, $v_i:=\langle \mathbf{v},e_i\rangle $. Then, for each coordinate $i\in\{1,\dots,n\}=\bm{S}$, the time-evolution of the system is described by
\begin{equation}
    ^C_0 D_t^{\alpha_i} u_i(t) = (\mathbf{A}\mathbf{u})_i,\quad \mathbf{u}(0)=\mathbf{u}_0\in\mathbb{R}^n,\; \mathbf{A}\in\mathbb{R}^{n\times n},\; \bm{\alpha}\in(0,1)^n
    \label{main}
\end{equation}
where the derivative in time is given by the \textit{Caputo fractional derivative} of order $\alpha\in(0,1)$
\begin{equation}
\label{caputo-derivative}
    ^C_0 D_t^{\alpha_i} f(t) = \frac{1}{\Gamma(1-\alpha_i)}\int_0^t ds\;\frac{f'(s)}{(t-s)^{\alpha_i}}
\end{equation}
which acts with a different fractional exponent in each coordinate. In Section \ref{stochastic-passes}, the set of Euclidean coordinates $\bm{S}$ corresponds in practice to the state space of a suitable random walk, so hereafter we shall also refer to $\bm{S}$ as the set of nodes. $\Gamma(\alpha)$ is the well-known gamma function\footnote{Letter $\Gamma$ only denotes the gamma function in this section. To be coherent with notation in \cite{ertel2022operationally}, which is one of the relevant references for this work, $\Gamma_t$ will denote the random walk process as defined in Section \ref{stochastic-forward}.} which extends the factorial to complex numbers. We denote the complete vector of fractional exponents as $\bm{\alpha}=\{\alpha_i\}_{i=1}^n\in(0,1)^n$. With a varying fractional exponent through space, we intend to model a more heterogeneous environment characterized physically by a spatially varying anomalous diffusion. In a discrete setting, this corresponds to assuming a different $\alpha_i$ in each coordinate $i$.

The Caputo derivative can be rewritten as the Laplace convolution of the classical differential operator with a fractional integral kernel\footnote{We use notation $I^\alpha$ instead of $J^\alpha$ for the Liouville integral due to future usage of $J$ as a functional $J_T$ over a path $\Gamma_t$, once again retaining notation from our main references.} \cite{gorenflo1997fractional}
\begin{equation}
\label{convolution-representation}
    ^C_0 D_t^{\alpha} f(t) = I^{1-\alpha} D f=\frac{t_+^{\alpha}}{\Gamma(1-\alpha)}*\frac{\partial f(t)}{\partial t}
\end{equation}
where the suffix $+$ indicates that the kernel vanishes for $t<0$. Swapping the order of the operators to $DI^{1-\alpha}$ results in a different formulation of the fractional derivative called \textit{Riemann-Liouville derivative}. Unlike both fractional derivatives, the fractional integral kernel $\phi_\alpha (t)=t^{\alpha-1}_+/\Gamma(\alpha)$ fulfills the semigroup property $\phi_\alpha (t)*\phi_\beta (t)=\phi_{\alpha+\beta} (t)$ as well as a convenient representation in Laplace space, where $\div$ denotes representation in Laplace space.
\begin{equation}
\label{mittag-laplace-representation}
    \phi_\alpha(t)\div \frac{1}{s^\alpha}\Longrightarrow I^{\alpha}f(t)\div \frac{\Tilde{f}(s)}{s^\alpha}
\end{equation}
We also make use of the following representations in Laplace space of the Mittag-Leffler function.
\begin{equation}
\label{laplace-pairs}
    E_\alpha(-\lambda t^\alpha)\div \frac{s^{\alpha-1}}{s^\alpha+\lambda},\quad t^{\beta-1}E_{\alpha,\beta}(-\lambda t^\alpha)\div\frac{s^{\alpha-\beta}}{s^\alpha+\lambda},\quad Re (s)>\vert \lambda \vert^{1/\alpha}
\end{equation}
$E_\alpha$ is an abbreviation for $E_{\alpha,1}$, as we can see in their Laplace space representation. For an in-depth exposition of these topics, see \cite{gorenflo1997fractional}.

The solution to the F-ODE (linear) system (\ref{main}) at time $t$ is given by $\mathbf{u}(t)=\mathbf{E}(t)\mathbf{u}(0)$, where $\mathbf{E}(t)$ is an unknown matrix. When the fractional exponent is homogeneous in space ($\bm{\alpha}=\alpha \mathbf{1}$), the unknown matrix is given by the Mittag-Leffler matrix function $
\mathbf{E}(t)=E_\alpha(\mathbf{A}t^\alpha)
$. However, when the fractional exponent is heterogeneous, there is no closed-form solution to the $\mathbf{E}(t)$
to the best of our knowledge.

Many of the existing numerical methods for an heterogeneous exponent attempt to discretize in time either the Caputo derivative (\ref{caputo-derivative}) as in the well-known L1 scheme \cite{yan2018analysis} or a Volterra reformulation of Equation (\ref{main}) as in \cite{diethelm2003efficient}, and then solve the underlying algebraic problem accordingly.

There is virtually no established research on explicit and scalable sensitivity computation for fractional linear models. Existing methods, such as reverse-mode automatic differentiation \cite{innes2018don,moses2021reverse} and the adjoint method \cite{maryshev2013adjoint,caputo2015duality,hendy2022cole,antil2020fractional}, avoid explicit computation of the sensitivity matrix and, despite being considered efficient, these approaches still face significant scalability challenges when addressing large-scale problems

We propose a stochastic method based on an unbiased estimator of a single row of the unknown matrix
\begin{equation}
    \label{matrix-row}
    \left[\mathbf{E}(t)_{ij} \right]_{1\leq j \leq n}
\end{equation}
which returns the scalar product with the initial condition to yield the solution at time $t=T$
\begin{equation}
    u_i(T)=\sum_{j=1}^n \mathbf{E}(T)_{ij}u_j(0)
\end{equation} 
Our algorithm also provides a computationally efficient technique capable of explicitly estimating the following four sensitivities with a negligible cost in terms of memory.
\begin{equation}
\label{sensitivities}
    \left[\frac{\partial u_i(T)}{\partial a_{jk}}\right]_{1\leq j,k\leq n} \quad\left[\frac{\partial u_i(T)}{\partial u_j(0)}\right]_{1\leq j\leq n}\quad \left[\frac{\partial u_i(T)}{\partial \alpha_{j}}\right]_{1\leq j\leq n} \quad\frac{\partial u_i(T)}{\partial T}
\end{equation}
(Notice that the $\partial u_i(T)/\partial u_j(0)$ sensitivity corresponds to $\mathbf{E}(t)_{ij}$.)

We believe this method combines the benefits of forward-mode differentiation, as it can output the sensitivity matrix itself, and backward-mode differentiation,  by allowing runtime multiplication of the sensitivity estimators at (\ref{sensitivities}) with any desired quantity. This multiplication is exceptionally fast because, as we will see, only the quantities involved in each random walk path have non-zero contribution to the sensitivities. For example, if, for a particular random walk, the spatial path is given by $I=\{S_1,S_2,\dots,S_{1+\nu_T}\}$, then, for all $j\notin I$, the estimator of $\partial u_i(T)/\partial \alpha_j$ is zero. Consequently, the randomized vector-Jacobian products derived from a reasonable number of random walks tend to be highly sparse. This sparsity can then be leveraged to achieve extremely fast stochastic gradient descent through the massive parameter space of the system.

\section{Description of the stochastic method}\label{stochastic-passes}

\subsection{Stochastic representation of the solution}
\label{stochastic-forward}
To derive a stochastic solution to Equation (\ref{main}), it is essential to rewrite first such an equation as an integral equation. This is done through the following theorem. 
\begin{theorem}
An integral representation to Equation (\ref{main}) with $a_{ii}\ne 0$ is given by
\begin{align}
\label{integral-representation}
&u_i(T)=\mathbb{P}(\xi_*\geq T)\frac{E_{\alpha_i}(a_{ii}T^{\alpha_i})}{\mathbb{P}(\xi_*\geq T)}u_i(0)+\nonumber\\&\sum_{k\neq i}\mathbb{P}(\mathcal{J}_i=k)\int_0^T d\tau \;p(\tau;\xi_*)\frac{-a_{ii}\tau^{\alpha_i-1}E_{\alpha_i,\alpha_i}(a_{ii}\tau^{\alpha_i})}{p(\tau;\xi_*)}\chi(i,k)u_k(T-\tau) 
\end{align}
where $p(\tau,\xi_*)$ and $\mathbb{P}(\xi_*\geq T)$ are respectively the density function and the survival function for an auxiliary random variable $\xi_*$ supported on $\mathbb{R}_{\geq 0}$ absolutely continuous with respect to the Lebesgue measure.

$\mathcal{J}_i$ is a discrete random variable taking values in all states $\{1,\dots,n\}\backslash \{i\}$ and describes the probability of jumping from state $i$ to any other state $k$
\begin{equation}
\mathbb{P}(\mathcal{J}_i=k)=\frac{\vert a_{ik}\vert}{\sum_{j\neq i}\vert a_{ij}\vert}=\frac{\mathrm{sgn}(a_{ik})a_{ik}}{\sum_{j\neq i}\mathrm{sgn}(a_{ij})a_{ij}}
\end{equation}
and $\chi(i,k)$ is a multiplicative signed variable updated according to the jump of the random walk 
\begin{equation}
    \chi(i,k)=-\frac{{\rm sgn}(a_{ik})\sum_{j\neq i}\vert a_{ij}\vert}{a_{ii}}=-\frac{\mathrm{sgn}(a_{ik})\sum_{j\neq i}\mathrm{sgn}(a_{ij})a_{ij}}{a_{ii}}
\end{equation}
\end{theorem}

\begin{proof}
We first split the F-ODE system to be solved (\ref{main}) as follows:
\begin{align}
    ^C_0 D_t^{\alpha_i} u_i(t) = \sum_{k\neq i}a_{ik}u_k(t) + a_{ii}u_i(t)
\end{align}
We now use representation (\ref{convolution-representation}) of the Caputo derivative together with the semigroup property of the fractional integral.
\begin{align}
I^{\alpha_i} (^C_0 D_t^{\alpha_i} u_i(t))&=I^{\alpha_i} I^{1-{\alpha_i}}D u=IDu=u_i(t)-u_i(0)\nonumber \\
&= \sum_{k\neq i}a_{ik}\phi_{{\alpha_i}}(t)*u_k(t) + a_{ii}\phi_{{\alpha_i}}(t)*u_i(t)
\label{to-laplace}
\end{align}

Transforming Equation (\ref{to-laplace}) into Laplace $s$-space and after some algebra, we obtain:
\begin{align}
    &\Tilde{u}_i(s)=\frac{u_i(0)}{s}+\sum_{k\neq i}a_{ik}\frac{\Tilde{u}_k(s)}{s^{\alpha_i}}+a_{ii}\frac{\Tilde{u}_i(s)}{s^{\alpha_i}}\nonumber\\
    &\Tilde{u_i}(s)\left(1-\frac{a_{ii}}{s^{\alpha_i}}\right)=\frac{u_i(0)}{s}+\sum_{k\neq i}a_{ik}\frac{\Tilde{u}_k(s)}{s^{\alpha_i}}\nonumber\\
    &\Tilde{u_i}(s)=u_i(0)\frac{s^{\alpha_i-1}}{s^{\alpha_i}-a_{ii}}+\sum_{k\neq i}a_{ik}\Tilde{u}_k(s)\frac{1}{{s^{\alpha_i}-a_{ii}}}
\end{align}
Using the Mittag-Leffler Laplace pairs (\ref{laplace-pairs}), we return to $t$-space and evaluate at $t=T$.
\begin{align}
u_i(T)&=E_{\alpha_i}(a_{ii}T^{\alpha_i})u_i(0)+\sum_{k\neq i}\int_0^T d\tau\;a_{ik}u_k(T-\tau)\tau^{\alpha_i-1}E_{\alpha_i,\alpha_i}(a_{ii}\tau^{\alpha_i})
\end{align}
Finally, each term of the equation is multiplied and divided by the survival and density function of $\xi_*$ respectively. There is no division by zero as we conveniently choose $\xi_*$ absolutely continuous with respect to the Lebesgue measure. For each summand $k\neq i$ being integrated, we multiply and divide by \begin{equation}
    \frac{-a_{ii}\;{\rm sgn}(a_{ik})}{-a_{ii}\sum_{j\neq i}\vert a_{ij}\vert}
\end{equation} to reach Equation (\ref{integral-representation}).
\end{proof}
\begin{remark}
    We write $\mathrm{sgn}(a_{ik})a_{ik}$ instead of $\vert a_{ik}\vert$ to highlight how to obtain $\vert a_{ik}\vert$ from $a_{ik}$ and to simplify taking derivatives with respect to $a_{ik}$.
\end{remark}

\begin{remark}
When $a_{ii}<0$ we can simplify Equation (\ref{integral-representation}) by taking the auxiliary random variable to be $\xi_*=\xi_i:=\xi(i)$ Mittag-Leffler distributed with rate $\lambda=-a_{ii}$ and exponent $\alpha_i$ following notation in the Laplace pairs (\ref{laplace-pairs}).
\begin{align}
\label{integral-representation-clearer}
u_i(T)=\mathbb{P}(\xi_i\geq T)u_i(0)+\sum_{k\neq i}\mathbb{P}(\mathcal{J}_i=k)\int_0^T d\tau \;p(\tau;\xi_i)\chi(i,k)u_k(T-\tau) 
\end{align}
This fact can be exploited each time the random walk enters a node $i$ with $-a_{ii}<0$ to greatly reduce computational time and variance.
\end{remark}

Through the integral equation in Equation (\ref{integral-representation}), the stochastic representation of the solution of Equation (\ref{main}) comes as follows. Let assume $a_{ii}<0$, then the solution $u_i(T)$ can be obtained
as the expected value of a functional $J_T$ acting over the paths of a continuous-time random walk \cite{ertel2022operationally,fu_conditional_1997} $\Gamma_t\colon \Omega\times [0,\infty)\rightarrow \bm{S}$ taking values in the discrete state space $\bm{S}=\{1,\dots,n\}$ with sojourn times given either by $\xi_*$ or $\{\xi_i\}_{i=1}^n$ and transition probabilities given by $\{\mathcal{J}_i\}_{i=1}^n$. The transition probabilities generate a Markovian sequence of states $\{S_j\}_{j=1}^\infty$ in the node space $\bm{S}$ and the sojourn times generate a sequence of non-negative numbers $\{\Delta \tau_j\}_{j=1}^\infty$, where each $\Delta \tau_j$ depends exclusively, in the simplified case, on $S_j$. Both sequences yield a Markov chain in the extended space $\bm{S}\times \mathbb{R}_{\geq 0}$ that fully characterizes each path of the random walk
\[
(\Gamma_t(\omega))_{t\in [0,\infty)}\equiv \{(S_j,\Delta \tau_j)\}_{j=1}^\infty
\]
and induces a density function \cite{ertel2022operationally} in the quotient pathspace $\Gamma_t\colon \Omega\times [0,T]\rightarrow \bm{S}$
\begin{equation}
    p(\Gamma_t(\omega))=\mathbb{P}(\xi(S_{1+\nu_T(\omega)})>\Delta \tau_{1+\nu_T}(\omega))\prod_{j=1}^{\nu_T(\omega)}p(\Delta \tau_j;\xi(S_j))\mathbb{P}(\mathcal{J}_{S_j}=S_{j+1})
    \label{p_gamma}
\end{equation}
where $\nu_t\colon \Omega \times [0,\infty)\rightarrow \mathbb{Z}_{\geq 0}$ is called the dual counting process to $\Gamma_t$. It indicates the number of jumps given by the path $\omega$ at time $t$, and it is formally defined as
\begin{equation}
\label{dual-jump-process}
    \nu_t(\omega)=\argmin_{i\in \mathbb{Z}_{\geq 0}} \left(\sum_{j=1}^i \Delta \tau_j \leq t\right)
\end{equation}
For the simplified case in Equation (\ref{integral-representation-clearer}), the solution to the F-ODE system at node $i$ and time $t=T$ is given by the following expected value over $\Omega$ of the paths $\Gamma_t$ generated by $\xi_j$ and $\mathcal{J}_j$ starting at node $S_1=i$.
\begin{equation}
\label{stochastic-representation-clearer}
    u_i(T)=\mathbb{E}(u_{\Gamma_T}(0)\prod_{j=1}^{\nu_T}\chi(S_j,S_{j+1})\mid S_1 = i):=\mathbb{E}(J_T(\Gamma_t)\mid \Gamma_0 = i)
\end{equation}
In the more general case, the stochastic representation is also given by the expected value over paths generated by $\xi_*$ and $\mathcal{J}_i$ (or any suitable collection of sojourn times $\{(\xi_*)_j\}$) with extra terms

\begin{align}
\label{stochastic-representation}
&u_i(T)=\mathbb{E}\left(
\frac{E_{\alpha_{S_{\nu_T}}}(a_{S_{\nu_T}S_{\nu_T}}(T-\sum_{j=1}^{\nu_T} \Delta \tau_j)^{\alpha_{S_{\nu_T}}})}
{\mathbb{P}(\xi_*\geq T-\sum_{j=1}^{\nu_T} \Delta \tau_j)} u_{\Gamma_T}(0) \times \right. \nonumber \\ 
&\left. \prod_{j=1}^{\nu_T} \frac{-a_{S_j S_j}(\Delta \tau_j)^{\alpha_{S_j}-1} E_{\alpha_{S_j},\alpha_{S_j}}(a_{S_j S_j}(\Delta \tau_j)^{\alpha_{S_j}})}{p(\Delta \tau_j;\xi_*)}
\chi(S_j,S_{j+1}) \mid S_1 = i 
\right)
\end{align}

\begin{remark}
    This technique is easily recognized as importance sampling in statistics, and has been extensive used in reinforcement learning \cite{sutton2018reinforcement} and Feynman-Kac type formulas for classical PDEs \cite{acebronNewParallelSolver2011a}.
\end{remark}
\begin{remark}
\label{nu_T}
    Notice how in Equation (\ref{stochastic-representation-clearer}) the functional $J_T(\Gamma_t)$ only depends on the states $S_1,S_2,\dots$ up to the state at time $T$,  $S_{1+\nu_T}$, and not explictly on the final time $T$ itself. This observation will become important later on when computing the time sensitivity.
\end{remark}

\subsection{Stochastic representation for the derivatives}
\label{sec:back}
For clarity, we present here the sensitivities obtained for the case of the stochastic solution in Equation (\ref{stochastic-representation-clearer}). Even though it is possible to compute all sensitivities for the general case in Equation (\ref{stochastic-representation}), the algebra is much more involved, and the results only differ slightly from the results obtained with Equation (\ref{stochastic-representation-clearer}), where some extra terms are added to the estimators coming from standard derivatives.  


For the four parameters (matrix $\mathbf{A}$, vectors $\mathbf{u}_0$ and $\bm{\alpha}$, and positive real number $T$) characterizing the F-ODE system, the sensitivity with respect to the vector $\mathbf{u}_0$ is the most straightforward to compute, since $\mathbf{u}_0\mapsto J_T(\omega;\mathbf{u}_0)$ is a differentiable function for a fixed path $\omega$, and then we can resort to the reparametrization trick. The closed form formula for this estimator, for each $j$-th component $(\mathbf{u}_0)_j$ of the vector, is given by 
\begin{align}
    \frac{\partial u_i(T)}{\partial (\mathbf{u}_0)_j}&=\mathbb{E}\left(\frac{\partial}{\partial (\mathbf{u}_0)_j} u_{\Gamma_T}(0)\prod_{j=1}^{\nu_T}\chi(S_j,S_{j+1})\mid S_1 = i\right)\nonumber\\
&=\mathbb{E}\left(\delta(S_{1+\nu_T},j\prod_{j=1}^{\nu_T}\chi(S_j,S_{j+1})\mid S_1 = i\right)
\end{align}
where $\delta(i,j)$ is the Kronecker delta. This sensitivity is involved in inferring the initial conditions from final observations at time $t=T$. Essentially, it corresponds to the row of the unknown matrix described in Section \ref{frac-lin-sys}.
\begin{equation}
 \nabla_{\mathbf{u}_0} u_i(T)=\left[\mathbf{E}(t)_{ij} \right]_{1\leq j \leq n}
\end{equation}
However, as it was pointed out in the previous section, the reparametrization trick does not apply to the remaining sensitivities and we need to make use of more sophisticated discrete event system differentiation techniques.
\subsubsection{Matrix and fractional exponent sensitivities}
\label{sec:mall}
The Malliavin weight technique \cite{warren2013malliavin}, also known as the score function method \cite{kleijnen1996optimization}, is a method to compute the sensitivity of the expected value of a functional $J_T$. In practice, this requires differentiating with respect to a given parameter $\theta$ the density probability of a path, provided such a density probability $p_\theta(\Gamma_t(\omega))$ depends on the parameter $\theta$.

This is the case of the sensitivities with respect to the off-diagonal elements of matrix $\mathbf{A}$, controlling the transition between states of the Markov chain, and both, the elements of the diagonal of $\mathbf{A}$ and $\bm{\alpha}$, governing the sojourn times. These sensitivities appear in image reconstruction~\cite{buonocore2018tomographic}, thermal diffusivity reconstruction~\cite{karashbayeva2023estimation}, and any graph learning method governed by a diffusive model \cite{thanou2017learning}. Notice however that the elements of $\mathbf{A}$ are also involved in the computation of $J_T(\Gamma_t)$, so in practice such a functional depends also on $\theta$. Hereafter, to simply notation we omit the subscript $T$ and highlight the $\theta$-dependency as $J_\theta$.

The Malliavin weight technique reads as follows
\begin{align}
\label{score-function}
    &\partial_\theta \int_\Omega d\omega\; p_\theta(\Gamma_t(\omega)) J_\theta (\Gamma_t(\omega))\nonumber\\&=\int_\Omega d\omega \;(\partial_\theta J_\theta (\Gamma_t(\omega))p_\theta(\Gamma_t(\omega))+ J_\theta (\Gamma_t(\omega))\partial_\theta  p_\theta(\Gamma_t(\omega)))\nonumber\\
&=\int_\Omega d\omega \;p_\theta(\Gamma_t(\omega))\left(\partial_\theta J_\theta (\Gamma_t(\omega))+ J_\theta (\Gamma_t(\omega))W_\theta(\Gamma_t(\omega))\right),
\end{align}
where the random variable 
\begin{equation}
    W_\theta(\Gamma_t)=\frac{\partial_\theta  p_\theta(\Gamma_t)}{p_\theta(\Gamma_t)}=\partial_\theta\log(p_\theta(\Gamma_t))
\end{equation}
is known as the Malliavin weight. This quantity is straightforward to compute for a fixed path $\omega$. In fact, from
Equation (\ref{p_gamma}) we readily have
\begin{align}
\label{score}
     \partial_\theta \log(p_\theta(\Gamma_t(\omega)))&=\sum_{j=i}^{\nu_T}\left[\partial_\theta \log(p(\Delta \tau_j;\xi(S_j)))+\partial_\theta \log(\mathbb{P}(\mathcal{J}_{S_j}=S_{j+1}))\right]\nonumber\\
    &+\partial_\theta\log(\mathbb{P}(\xi(S_{1+\nu_T})>T-\sum_{j=1}^{\nu_T}\Delta \tau_j))
\end{align}
To simplify the formulas, Equation (\ref{score-function}) can be rewritten as 
\begin{align}
&\int_\Omega d\omega \;p_\theta(\Gamma_t(\omega)\left(\partial_\theta J_\theta (\Gamma_t(\omega))+ J_\theta (\Gamma_t(\omega))W_\theta(\Gamma_t((\omega)))\right)\nonumber\\
&=\int_\Omega d\omega \;p_\theta(\Gamma_t)J_\theta (\Gamma_t)\left(\partial_\theta \log(J_\theta (\Gamma_t))+ \partial_\theta\log(p_\theta(\Gamma_t))\right),
\end{align}
which corresponds to the following expected value
\begin{align}
\mathbb{E}(J_\theta (\Gamma_t)(\partial_\theta \log(J_\theta (\Gamma_t))+ \partial_\theta\log(p_\theta(\Gamma_t)))\mid S_1=i).
\end{align}
Moreover, from Equation (\ref{stochastic-representation-clearer}) it holds that
\begin{align}
\label{extra-score}
    \partial_\theta \log(J_\theta (\Gamma_t))=\sum_{i=1}^{\nu_T}\partial_\theta\log(\chi(S_i,S_{i+1})).
\end{align}

We now provide closed-forms for all log-derivatives. Note that $\chi(i,k)$ only depends on the elements of the $i$-th row of matrix $\mathbf{A}$, therefore only the derivative with respect to those elements is non-zero.
\begin{equation}
\frac{\partial}{\partial a_{ij}}\log(\chi(i,k))=\begin{cases}
    &\frac{{\rm sgn}(a_{i j})}{\sum_{l\neq i}\vert a_{i l}\vert } \quad j\neq i\\
    &\frac{-1}{a_{i i}}\quad j=i
\end{cases}
\end{equation}
Concerning the log-derivative of the probability function $\mathbb{P}(\mathcal{J}_i=k)$ in Equation (\ref{score}), it turns out that only the derivative with respect to off-diagonal entries of the $i$-th row of matrix $\partial u_i(T)/\partial \mathbf{A}$ are non-zero when in state $i$. This is because the Markov chain has no self-loops, and consequently the probability of jumping from a state to itself is zero. Thus,  
\begin{equation}
    \frac{\partial}{\partial a_{ij}}\log(\mathbb{P}(\mathcal{J}_i=k))=\delta_{jk}\frac{{\rm sgn}(a_{ik})}{\vert a_{ik}\vert}-\frac{{\rm sgn}(a_{ij})}{\sum_{l\neq i}\vert a_{il}\vert},\quad i\neq j
\end{equation}
Summing both terms yields, 
\begin{equation}
\frac{\partial}{\partial a_{ij}}\log(\chi(i,k)) + \frac{\partial}{\partial a_{ij}}\log(\mathbb{P}(\mathcal{J}_i=k)) = 
\begin{cases}
    \frac{1}{a_{ik}} & \text{if } j = k \\
    \frac{-1}{a_{ii}} & \text{if } j = i \\
    0 & \text{otherwise,}
\end{cases}
\end{equation}
which is an efficient formula to be computed at each step.
The derivatives involved in the sojourn times are straightforward to compute and are described in the pseudocode at Subsection \ref{implementation}.
\begin{remark}
\label{remark-malliavin}
The usage of the density function (\ref{p_gamma}) requires some form of underlying Euclidean space. From the perspective of computer simulation, a path in $[0,T]$ requires an arbitrary-length chain of random numbers available. This corresponds to the underlying probability space being $\Omega=\otimes_{j=1}^\infty [0,1]$ which has a probability measure but not a proper density function.

We can check the validity of the formulas by applying the log-derivative trick to Equation (\ref{integral-representation-clearer}). By choosing $\theta=\alpha_j$, the solution to the following recursion is equivalent to the result obtained by the Malliavin weight technique.
\begin{align}
&\partial_{\alpha_j} u_i(T)=\mathbb{P}(\xi_i\geq T)\partial_{\alpha_j} \log(\mathbb{P}(\xi_i\geq T))u_i(0)+\nonumber\\&\sum_{k\neq i}\mathbb{P}(\mathcal{J}_i=k)\int_0^T d\tau \;p(\tau;\xi_i)\chi(i,k)\left(\partial_{\alpha_j} \log(p(\tau;\xi_i))u_k(T-\tau)+\partial_{\alpha_j}u_k(T-\tau)\right)
\end{align}
Any other choice of parameters will yield the same results. For a justification from a probabilistic point of view, we refer to Appendix \ref{appendix-B}.

\end{remark}
\subsubsection{Time sensitivities}
Time sensitivities are used when we do not know how long an event has been taking place, and we need to infer it from observations. For example, in \cite{thanou2017learning}, a model of multiscale dynamics of diffusion over a graph is used, which requires the time sensitivities. The main issue with this sensitivity is that a small perturbation $T+\Delta T$ does not change the density function, it enlarges the pathspace itself, so the previous Malliavin weight arguments do not apply. The sought sensitivity is equivalent, by definition, to the following limit taken from positive time increments $\Delta T>0$ due to $t\mapsto u_i(t)$ being differentiable.
\begin{equation}
    \partial_T u_i(T)=\lim_{\Delta T \shortarrow{7} 0}\frac{1}{\Delta T}\mathbb{E}(J_{T+\Delta T}(\Gamma_t)-J_{T}(\Gamma_t)\mid \Gamma_0 = i)
\end{equation}
If we fix a path $\omega$ and take the limit $\Delta T \shortarrow{7} 0$ inside the expected value with respect to $\omega$, we obtain a null estimator due to random walks being càdlàg piecewise constant.
\begin{equation}
\mathbb{E}\left(\lim_{\Delta T \shortarrow{7} 0}\frac{1}{\Delta T}(J_{T+\Delta T}(\Gamma_t(\omega))-J_{T}(\Gamma_t(\omega)))\right)  =\mathbb{E}(0)=0.
\end{equation}
However, this is clearly wrong, since the solution $\mathbf{u}(t)$ to the original F-ODE system~(\ref{main}) is not constant over time unless it is already in a steady state. This is another example of the common failure of the reparametrization trick when used for discrete event systems.
\begin{theorem} An unbiased estimator for the time sensitivity $\left.\frac{\partial u_i(t)}{\partial t}\right |_{t=T}\colon = \partial_T u_i(T)$ is given by
    \begin{align}
    \partial_T u_i(T)=\mathbb{E}\left(\frac{p(T-\sum_{i=1}^{\nu_T} \Delta \tau_i;\xi(S_{1+\nu_T}))}{\mathbb{P}(\xi(S_{1+\nu_T})>T-\sum_{i=1}^{\nu_T} \Delta \tau_i)}(J_{2+\nu_T}(\Gamma_t)-J_{1+\nu_T}(\Gamma_t))\right)
\end{align}
where we subscript the functional $J_T$ by an integer, as its value only depends on $\nu_T$ (Remark \ref{nu_T}). We then artificially extend the random walk with a new state $S_{2+\nu_T}$ and weight the difference in functionals by the so called critical rate.
\end{theorem}
\begin{proof}
    Our proof can be found in Appendix \ref{appendix-A}. A more general proof can be seen in Chapter 3 in \cite{fu_conditional_1997} in the wider framework of generalized semi-Markov processes.
\end{proof}

\subsection{Variance bounds}
\label{variance-bounds}

Under a numerical point of view, no stochastic method is useful when the variance is unbounded. This is because it is standard practice to assume a bounded variance when constructing confidence intervals based on the classical central limit theorem (CLT).

We will establish first an upper bound for the variance of the stochastic solution in Equation (\ref{stochastic-representation}), for exponentially distributed sojourn times $\xi_*\sim p(\xi)=\lambda e^{-\lambda \xi}$ for $\lambda = \max_i \vert a_{ii}\vert$. Then, we will subsequently apply these results to the solution in  Equation (\ref{stochastic-representation-clearer}) by taking the exponential random variable as an upper bound to all Mittag-Leffler distributed sojourn times. To bound the variance in both cases, we bound the term $\mathbb{E}(J_T(\Gamma_t)^2)$ by conditioning on the number of jumps.
\begin{align}
    &\mathbb{V}(J_T(\Gamma_t))=\mathbb{E}(J_T(\Gamma_t)^2)-\mathbb{E}(J_T(\Gamma_t))^2\nonumber\\&\leq \mathbb{E}(J_T(\Gamma_t)^2)=\sum_{k=0}^\infty \mathbb{P}(\nu_T(\Gamma_t)=k) \mathbb{E}(J_T(\Gamma_t)^2\mid \nu_T(\Gamma_t)=k)
\end{align}
The distribution of the dual counting process at time $T$ for $\lambda$ rate exponentially distributed sojourn times is Poisson distributed with rate $\lambda T$.
\[
\mathbb{P}(\nu_T = k)=\frac{e^{-\lambda T}(\lambda T)^k}{k!}
\]
Then, by conditioning over $k$ jumps made at time $T$, we roughly bound the value $J_T(\Gamma_t)$. To do so, we first bound the multiplicative potential $\chi(i,k)$ by its maximum value over all pairs of states $1\leq i,k\leq n$, $M_\chi=\max_{i\neq k}\chi(i,k)$. We then pick a rough bound for all values of the importance sampling weight 
\[
M_p = \max_{1\leq i\leq n,t\in [0,T]}\frac{-a_{ii}t^{\alpha_i-1} E_{\alpha_{i},\alpha_{i}}(a_{i i}t^{\alpha_i})}{\lambda e^{-\lambda t}},\quad M_s = \max_{1\leq i\leq n,t\in [0,T]}\frac{E_{\alpha_{i}}(a_{i i}t^{\alpha_i})}{e^{-\lambda t}}
\]
which is finite due to $\bm{S}\times[0,T]$ being compact and both densities being absolutely continuous and bounded.
\begin{align}
\label{variance}
    &\mathbb{E}(J_T(\Gamma_t)^2)\leq \sum_{k=0}^\infty \frac{e^{-\lambda T}(\lambda T)^k}{k!} M_p^{2} (M_s M_\chi)^{2k}\max_{k}(\mathbf{u}_0)_k^2  \nonumber\\
    &=\max_{k}(\mathbf{u}_0)_k^2 \;M^2_p\; \sum_{k=0}^\infty \frac{e^{-\lambda T}((M_s M_\chi)^{2} \lambda T)^k}{k!}\nonumber\\
    &=\max_{k}(\mathbf{u}_0)_k^2 \;M^2_p\; e^{\lambda T ((M_s M_\chi)^2-1)}<\infty
\end{align}
For the particular case of the solution in Equation (\ref{stochastic-representation-clearer}), there are no importance sampling weights, that means $M_p=M_s=1$. As we pick $\lambda \geq \vert a_{ii}\vert$, and for all $\alpha\in(0,1)$, all Mittag-Leffler sojourn times are dominated by the exponential sojourn time with the same rate, we can also use the bound $$\mathbb{P}(\nu_T=k)\leq \frac{e^{-\lambda T}(\lambda T)^k}{k!}$$ and obtain a bound for the variance of Equation (\ref{stochastic-representation-clearer})
\begin{equation}
\mathbb{E}(J_T(\Gamma_t)^2)\leq \max_{k}(\mathbf{u}_0)_k^2 \; e^{\lambda T (M_\chi^2-1)}<\infty
\end{equation}
Note that when $|M_\chi|<1$, the variance does not increase with time. If $a_{ii}<0$ this occurs when the matrix is diagonally dominant, that is
$$
\sum_{j\neq i}\vert a_{ij}\vert<\vert a_{ii}\vert \quad \forall i.$$
The same arguments can be applied for bounding the sensitivities by choosing suitable bounds for the Malliavin weights and the critical rate in the compact interval $[0,T]$.

\subsection{Implementation}
\label{implementation}
The algorithm described below has the same structure of any generic embarrassingly parallel Monte Carlo method. In fact, we generate many independent paths using different cores/processors and then join together all the results through a suitable reduction, obtaining therefore an estimation of the sought quantities through the corresponding average. In the following it is described the pseudocode we used for generating a single random walk.
\begin{algorithm}[H]
\caption{Parallel implementation of algorithm}
\begin{algorithmic}[1]
\Statex Inputs: $\mathbf{A}\in\mathbb{R}^{n\times n}$, $\mathbf{u}_0,\bm{\alpha} \in\mathbb{R}^n$, $T\in\mathbb{R}_{\geq 0}$,  $i\in[1,\dots n]$, $N_s\in \mathbb{N}$, cumulative sum $\mathbf{P}\in\mathbb{R}^{n\times n}$ of embedded Markov chain.
\For{sim=1:$N_s$}
    \State $u_i(T) \gets 1$
    \State $\mathbf{W}_A,\mathbf{W}_\alpha \gets \mathbf{0}^{n\times n},\mathbf{0}^{n}$
    \State Generate $t\gets \tau \sim \xi(S_i)$ \Comment{Compute sojourn time}
    \While{$t < T$}
        \State $k \gets \min_j \text{ such that } \rho < \mathbf{P}_{ij}, \quad \rho\sim \mathcal{U}(0,1)$ \Comment{Next state}
        \State $u_i(T) \gets u_i(T) \times \text{sgn}(a_{ik}) \times \sum_{j\neq i}\vert a_{ij}\vert / (-a_{ii})$ \Comment{Update solution}
        \State Update $\mathbf{W}_A,\mathbf{W}_\alpha$
        \State Update state $i\gets k$ and time $t\gets t+\xi(S_i)$
    \EndWhile
    \State Update $\mathbf{W}_A,\mathbf{W}_\alpha$
    \State $\partial u_{i}(T)/\partial (\mathbf{u}_0)_i \gets u_i(T)$
    \State $u_i(T)\gets u_i(T) \times (\mathbf{u}_0)_i $
    \State $\partial u_{i}(T)/\partial \mathbf{A} \gets \mathbf{W}_A \times u_i(T)$
    \State $\partial u_{i}(T)/\partial \bm{\alpha} \gets \mathbf{W}_\alpha \times u_{i}(T)$
    \State Generate new state $k$ and compute $W_T$ with Equation \ref{eq:critical-rate}.
    \State $\partial u_{i}(T)/\partial T \gets  u_{i}(T) \times( (\text{sgn}(a_{i,k}) \times \sum_{j\neq i}\vert a_{ij}\vert / (-a_{i,i}))\times(\mathbf{u}_0)_k - (\mathbf{u}_0)_i)\times W_T  $ 

\EndFor
\end{algorithmic}
\end{algorithm}
To generate the random transition between states, we used the cumulative sum of matrix $\mathbf{P}$ with $(i,j)$-th entry given by $P_{ij}=\vert{a_{ij}}\vert/\sum_{j\neq i}\vert a_{ij}\vert$, searching for the first index $k$ such that $\sum_{j=1}^k P_{ij}>\rho$, where $\rho$ is a uniformly distributed random number in $[0,1]$. We describe in the following algorithm how the Malliavin weight is updated in practice. This has to be done together with the derivative of the functional over the path $J(\Gamma_t)$ itself.
\begin{algorithm}[H]
\caption{Update of weights $\mathbf{W}_A,\mathbf{W}_\alpha$}    
\begin{algorithmic}[1]
\Statex Inputs: $\mathbf{W}_A,\mathbf{W}_\alpha$, $\mathbf{A}$, $\bm{\alpha}$ current state $i$, next state $k$, last sojourn time $\tau$, total time $t$, final time $T$.
\If{$t < T$}
    \State $(\mathbf{W}_A)_{ik} \gets (\mathbf{W}_A)_{ik} + \frac{1}{a_{ik}}$
    \State $(\mathbf{W}_A)_{ii} \gets (\mathbf{W}_A)_{ii} - \frac{2}{a_{ii}} + \frac{\partial}{\partial a_{ii}} \log \left( E_{\alpha_i, \alpha_i}(a_{ii} \tau^{\alpha_i}) \right)$
    \State $(\mathbf{W}_{\bm \alpha})_i \gets (\mathbf{W}_{\bm \alpha})_i + \frac{\partial}{\partial \alpha_i} \log \left( E_{\alpha_i, \alpha_i}(a_{ii} \tau^{\alpha_i}) \right) + \log(\tau)$
\Else
\State $(\mathbf{W}_{\bm \alpha})_i \gets (\mathbf{W}_{\bm \alpha})_i + \frac{\partial}{\partial \alpha_i} \log \left( E_{\alpha_i}(a_{ii} (T-t+\tau)^{\alpha_i}) \right)$
\State $(\mathbf{W}_A)_{ii} \gets (\mathbf{W}_A)_{ii} + \frac{\partial}{\partial a_{ii}} \log \left( E_{\alpha_i}(a_{ii} (T-t+\tau)^{\alpha_i}) \right)$
\EndIf
\end{algorithmic}
\end{algorithm}

\subsection{Time complexity of the algorithm}

In the following, we discuss the computational cost of the algorithm, focusing merely on the execution time for generating a single random walk. For computing the stochastic solution, the most time-consuming parts of the algorithm are twofold: the time needed to generate the sojourn times per simulation, which grows linearly with the expected number of jumps, and the time spent to perform a linear search in each row to choose the node to jump to, which depends on the size of the matrix $\mathbf{A}$.

\begin{figure}[b]
    \centering
    \includegraphics[scale=0.45]{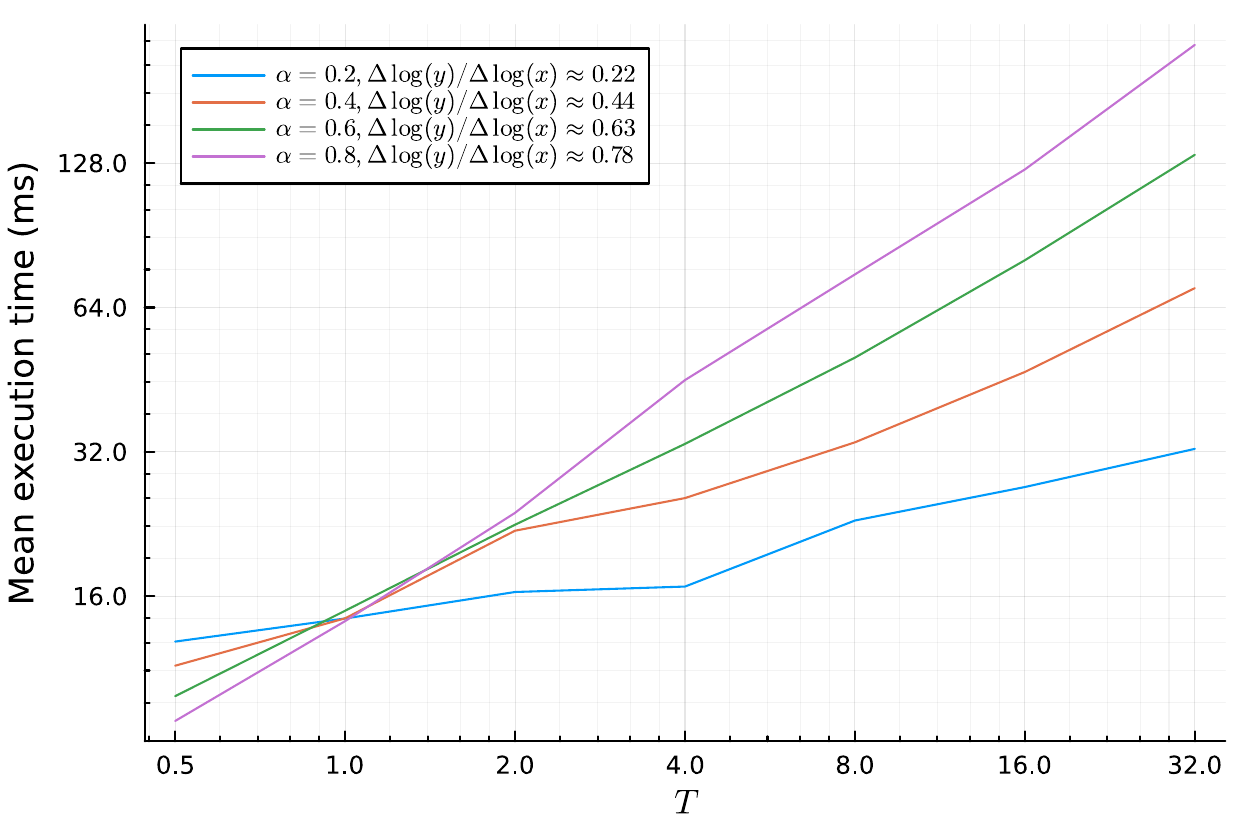}
    \caption{Power law behaviour of average execution time of a random walk with respect to final time $T$ for different values of the $\alpha$ exponent, being $\alpha_i=\alpha, \forall i$.}
    \label{fig:alphascale}
\end{figure}

Including sensitivity analysis increases the number of computations per random walk, with the evaluation of the Mittag-Leffler function being identified as the main time-consuming part. 
Note that evaluating this function is mandatory in order to update conveniently both the Malliavin weights for $\bm{\alpha}$ and $a_{ii}$, and for the CMC estimator for $T$.
We make use of a Laplace transform method \cite{garrappa2018computing} to compute $E_\alpha(z)$ and $E_{\alpha,\alpha}(z)$, since we have experimentally observed to be the fastest among all existing methods. Furthermore, we have observed that the expected execution time for the evaluation of the Mittag-Leffler function at Mittag-Leffler distributed random times (given by the following density function)
\begin{equation}
\label{pdf_f}
f(\tau)=\lambda \tau^{\alpha-1}E_{\alpha,\alpha}(-\lambda \tau^\alpha)
\end{equation}
remains constant through all possible values of $\alpha$ and $\lambda$. This implies, in practice, that the execution time for a particular random walk only depends on the number of calls to the Mittag-Leffler function and not on the particular spatial path followed by the random walk. The number calls is proportional to the number of jumps given by the random walk in $[0,T]$. This is the random variable $\nu_T$ given by Equation (\ref{dual-jump-process}), and, for fixed $\lambda$ and $\alpha$, is known to follow a fractional Poisson distribution \cite{laskin2003fractional} with expected value $\lambda T^\alpha$. 

In fact, numerical experiments confirm this finding, and this is shown in Fig.~\ref{fig:alphascale} in log-log scale, being the slopes of the curves (shown in the legend of the figure) in agreement with the theoretical estimates. For this experiment we consider a fixed 1D-Laplacian matrix and we average the time taken to simulate a random walk and its sensitivities in $[0,T]$ for different values of $T$ and a constant fractional exponent $\bm{\alpha}=\alpha\mathbf{1}$ for different values of $\alpha$. This is done through the Julia library \texttt{BenchmarkTools.jl} which adaptively chooses the number of simulations $N_s$ to take the average. Note that when $\alpha$ goes  to zero, we observe the typical subdiffusive behaviour characterized by few spatial jumps. 

\begin{figure}[b]
    \centering
    \begin{subfigure}[b]{0.45\textwidth}
        \centering
        \includegraphics[width=\textwidth]{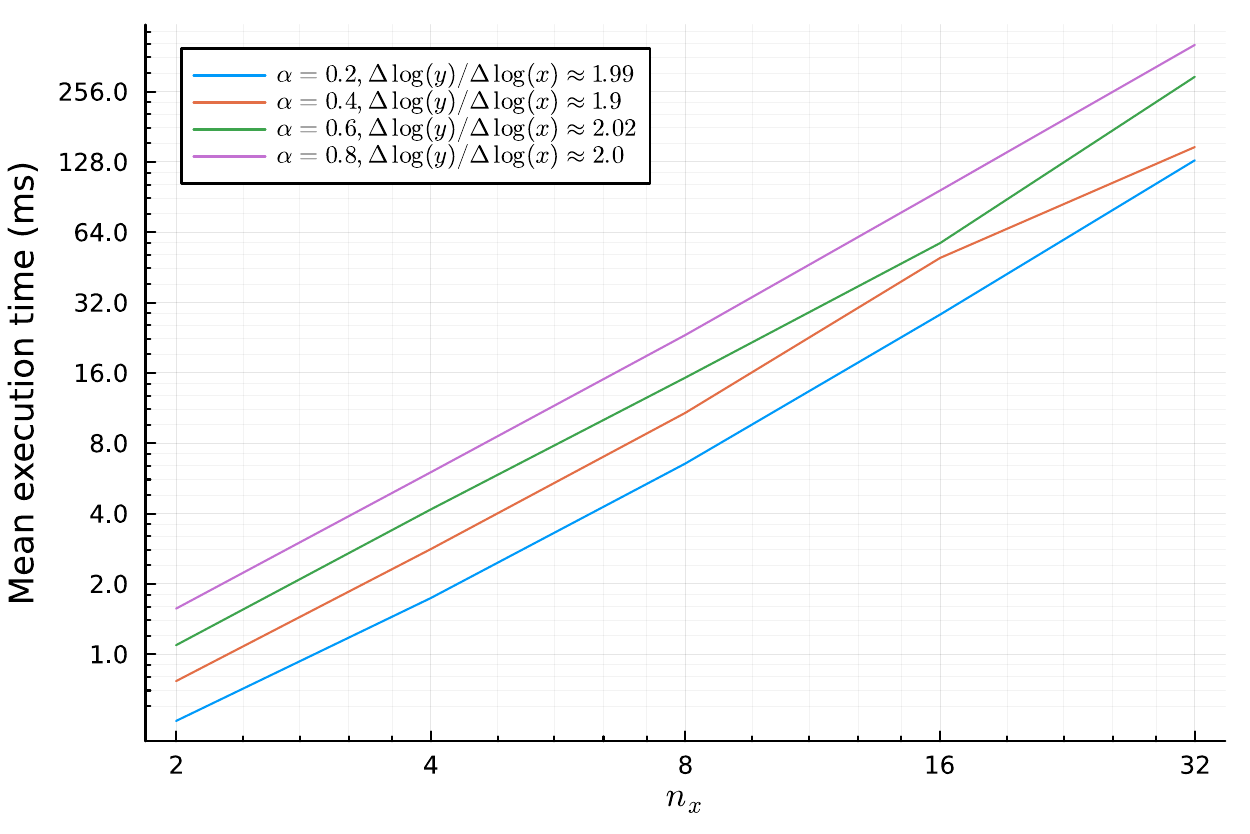}
        \caption{Mean execution time as function of the discretization points $n_x$ for a 1D-Laplacian matrix $L$, keeping fixed the time to $T=8$.}
        \label{fig:scalenx}
    \end{subfigure}
    \hspace{0.05\textwidth}
    \begin{subfigure}[b]{0.45\textwidth}
        \centering
        \includegraphics[width=\textwidth]{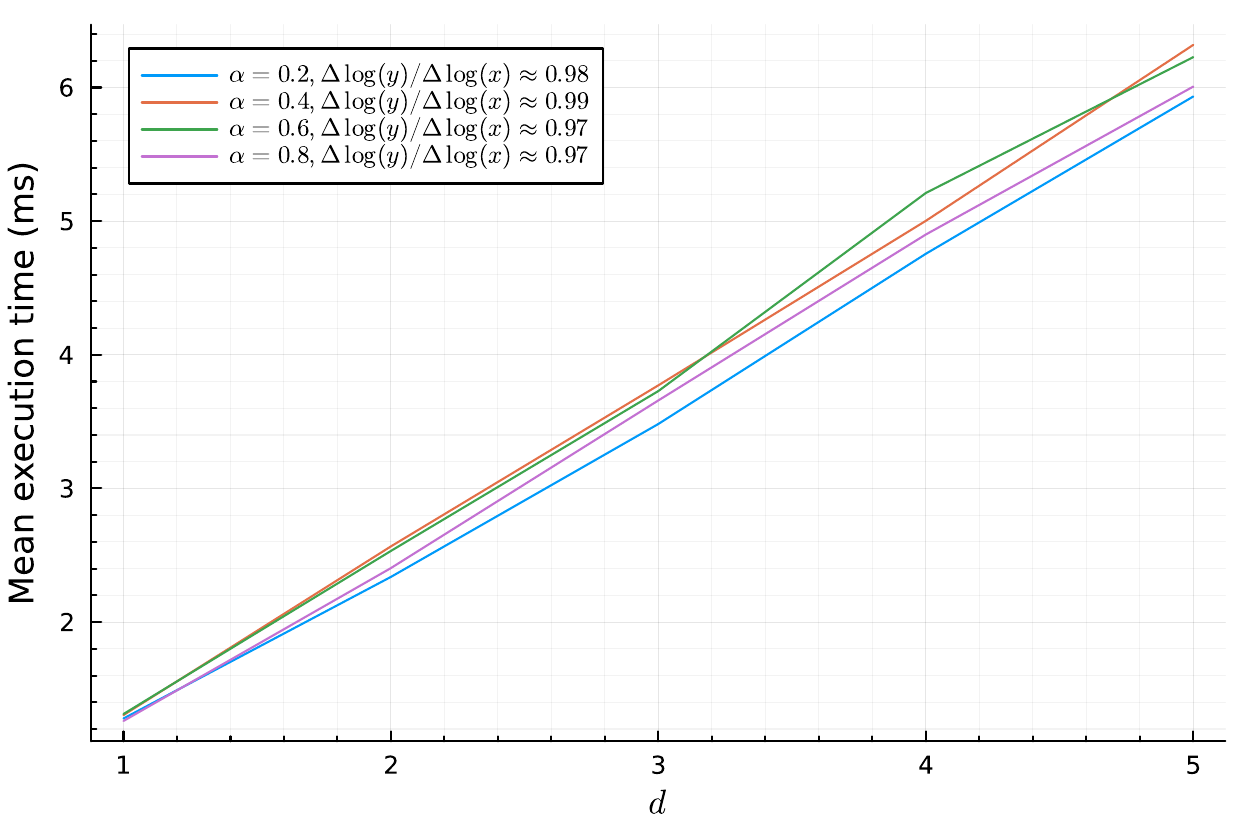}
        \caption{Mean execution time as function of the dimension $d$ of a $d-$dimensional Laplacian matrix $L$. Rest of parameters are kept fixed to $T=1$, and $n_x=4$.}
        \label{fig:scaled}
    \end{subfigure}
\end{figure}

It is worth observing in Equation (\ref{pdf_f}) the difference in the dependency of the fractional exponent $\alpha$ on both variables, the time $t$ and rate $\lambda$. As a consequence, a different contribution to the computational time of the algorithm is expected. Concerning the rate $\lambda$, recall from Equation (\ref{integral-representation}) that this corresponds to the diagonal elements of the matrix $\mathbf{A}$, $a_{ii}$. Hence we argue that the computational time should grow linearly with $a_{ii}$, since the expected value of the density probability in Equation~(\ref{pdf_f}) is proportional to $\lambda$.
To illustrate this behavior we consider here a simple test case. This consists in a matrix $\mathbf{A}$ corresponding to the discretization of the 1D Laplace operator using the standard 3-point stencil finite difference approximation, after discretizing in space with $n_x$ grid points. For this case $a_{ii}$ is proportional to $n_x^2$, and therefore the execution time should be proportional to $n_x^2$ accordingly. This is indeed what happens, as it can be seen in Fig.~\ref{fig:scalenx}.

In the case of problems where the matrices correspond to the discretization of the Laplacian operator, it is worth analyzing the time complexity of the algorithm with respect to the dimensionality of space. To this purpose we consider now the discretization of the $d-$dimensional Laplace operator using a $2d+1$-point stencil finite difference approximation. In Fig. \ref{fig:scaled} it is shown that the mean execution time grows linearly with the dimension, and independently of the value of $\alpha$.  This is in contrast to the typical exponential growth found in deterministic methods.

Note that the results shown in Figs.~\ref{fig:alphascale}, \ref{fig:scalenx} and~\ref{fig:scaled} correspond to F-ODE systems where the fractional exponent vector $\alpha_i = \alpha,\forall i$, and this is basically due to the fact that  the distribution of the sojourn times are the same. For a more general case the distribution of the sojourn times becomes different for each node, and thus we can only provide worst-case scenario complexity bounds for such a F-ODE system with parameters $\bm\theta=(\mathbf{A},\mathbf{u}_0,\bm{\alpha},T)$. This can be done by assuming all jumps follow the dominating fastest sojourn times among all nodes, which is given by
\begin{equation}
    \label{scaling-time}
     \max_{i,j} \mathcal{O}(a_{ii}T^{\alpha_j})
\end{equation}
This is a remarkable finding because it shows that the time complexity of the algorithm does not depend explicitly on the size of the matrix, but only on the underlying stochastic dynamics. 
Moreover, we believe Equation (\ref{scaling-time}) provides an insightful connection on how typical stability issues for deterministic methods, such as stiffness of the matrix, translate here in larger execution times.

Finally, to conclude, it is worth pointing out that the number of simulations we can run per unit of time increases linearly with the number of cores used \cite{acebron2019monte, guidotti2024fast}. Therefore, knowing the expected execution time for one random walk, we can infer the expected execution time for $r$ i.i.d. random walks distributed through $p$ cores. This is because as it was previously mentioned, the proposed algorithm is based on a embarrassingly parallel Monte-Carlo method. Due to the central limit theorem, it turns out that the size of the confidence intervals for the Monte-Carlo solution is proportional to $\sqrt{N_s}$, being $N_s$ is the number of random walks. Therefore, the size of the confidence interval can be increased by a factor of $\sqrt{p}$ within the same computational time
,  simply resorting to simulate in parallel the algorithm using $p$ cores.

\section{Numerical results}\label{numeric-results}
\subsection{Numerical validation}

To validate the theoretical results, we randomly generate $100$ F-ODE systems. We compute their solution and sensitivities both, deterministically and stochastically, to test whether the expected values of the stochastic method coincides with the results obtained by the deterministic method. This is done via null-hypothesis testing. The algorithms where coded in Julia v.1.10 and executed in a system powered by an AMD EPYC 9554P 64-Core Processor with 128 threads. The deterministic solution was computed using the L1 scheme \cite{yan2018analysis}. This scheme discretizes the timespan $[0,T]$ by $t_0=0<t_1<\dots<t_{N_t-1}<t_{N_t}=T$ and formally interpolates the unknown solution $\mathbf{u}(t)$ of Equation (\ref{caputo-derivative}) in the timegrid by a linear spline to yield an iterative solution for $\mathbf{u}(t)$ in which $\mathbf{u}(t_i)=f(\mathbf{u}(t_0),\dots,\mathbf{u}(t_{i-1}))$ for some function $f$. 

The sensitivities of the system were computed deterministically resorting to a 2-point stencil centered finite difference scheme \cite{frames_white_2023_10059774}. A suited Julia library is used to evaluate the function in the stencil $[-h,h]$, with $h=\mathcal{O}(\sqrt{\varepsilon})$ where $\varepsilon\approx 2.2\times 10^{-16}$ is 64-bit machine epsilon, which minimizes the sum of the discretization bias and the round-off error caused by choosing an $h$ too close to $\varepsilon$. This requires precisely $2n^2+4n+3$ calls to the L1 solver per test. Each call to the solver is computationally costly, since we should choose a value of $\Delta t$ small enough to ensure stability and minimize the discretization error as much as possible. Due to the large amount of computational time the deterministic solver requires, we constrain our numerical experiments to analyze random F-ODE systems of size $n=5$.

The matrices $\mathbf{A}^{(k)}$ for each test are generated by drawing independent realizations of the following random matrix.
\begin{equation}
\label{random-matrix}
    a^{(k)}_{ij} = \begin{cases}
        \mathcal{N}(0,1),\quad &i\neq j\\
        -(1+\mathcal{U}(0,1))\sum_{l\neq i}\vert a_{il}\vert\quad &i=j
    \end{cases}
\end{equation}
This choice ensures the matrices are diagonally dominant with negative diagonal while remaining as general as possible. Each individual component of the initial condition $\mathbf{u}_0$, as well as the final time $T$, are independently drawn a $\mathcal{U}(0,1)$ distributions. The fractional exponents $\bm{\alpha}$ are independently drawn from $\mathcal{U}(0.6,1)$ to ensure stability of the L1 scheme for a fixed $\Delta t$. 

To validate our solution we use T-tests in the scalar case and the Hotelling $T^2$-test using \texttt{HypothesisTests.jl}\footnote{https://github.com/JuliaStats/HypothesisTests.jl}, flattening the matrix $\frac{\partial u_i(T)}{\partial a_{ij}}$ to a vector. Our null-hypothesis is that the expected value of our Monte-Carlo method is the deterministic solution, as in $H_0\colon \mathbb{E}(J(\Gamma))=u_i(T)$ and so on, with $\frac{\partial u_i(T)}{\partial a_{ij}}$, $\frac{\partial u_i(T)}{\partial \alpha_j}$, $\frac{\partial u_i(T)}{\partial u_j(0)}$ and $\frac{\partial u_i(T)}{\partial T}$. The main idea of the tests is to construct $95\%$ confidence intervals for the Monte-Carlo solution and check whether the deterministic solution is within those confidence intervals. When this happens, we conclude that the tests were passed. If approximately $95\%$ of the tests are passed (it should be $95\%$ asymptotically when the number of tests goes to infinity), then there is not enough evidence to reject the null-hypothesis.

In Table \ref{table_results}, we summarized our findings, which confirm experimentally that the method indeed works.

\begin{table}[h]
\centering
\renewcommand{\arraystretch}{1.5} 
\setlength{\tabcolsep}{15pt} 
\begin{tabular}{|c|c|c|c|c|c|}
\hline
$n=5$ & $u_i(T)$ & $\frac{\partial u_i(T)}{\partial a_{ij}}$ & $\frac{\partial u_i(T)}{\partial \alpha_i}$ & $\frac{\partial u_i(T)}{\partial u_i(0)}$ & $\frac{\partial u_i(T)}{\partial T}$ \\
\hline
Tests passed over 100 & 94 & 94 & 93 & 94 & 95 \\
\hline
\end{tabular}
\caption{Number of experiments passing the null-hypothesis testing over 100 tests with $N_s=10^5$ random walks. Sensitivities were evaluated at the random parameters generated by Equation (\ref{random-matrix}).}
\label{table_results}
\end{table}
\begin{remark}
    The usage of null-hypothesis testing provides a non-rigorous evidence of a stochastic method working. The null-hypothesis tests only work asymptotically with $N_s \rightarrow \infty$ and, in our case, are actually flawed, as we only have access to a numerical approximation of the true solution $\mathbf{u}(t)$ and its sensitivities. We circumvent these issues by choosing $\Delta t$ small enough such that the numerical error is orders of magnitude smaller than the sample variance and a suited $N_s$. All sensitivity computations are correlated, so if a test fails for $u_i(T)$, it might correlate with any sensitivity test also failing. A more effective way of doing the tests would be to join all tests into a single $T^2$-test by flattening all sensitivities and solution onto a vector, but we opted to split up the tests merely for exposition purposes.
\end{remark}

\subsection{Test case: Robin boundary conditions}
To test the algorithms in a more real, although simplified, setting,  we consider the fractional in time heat equation in $[0,1]$ with mixed boundary conditions. The initial condition is a Gaussian pulse centered at $\mu=0.1$ and standard deviation $\sigma=0.025$, and for the variable fractional exponent we choose $g(x;\alpha)=\alpha(\sin(\pi x)+1)/4+0.5$. This depends on a single arbitrary parameter $\alpha$ which in the following we set equal to $\alpha_0=0.7$.
\begin{align*}
    &^C_0 D_t^{g (x;\alpha)}u=\kappa(x)\frac{\partial^2 u}{\partial x^2}, \qquad (x,t)\in(0,1)\times (0,T)\nonumber\\
    &u(x,0)=\exp(-(x-\mu)^2/(2\sigma^2))/C\nonumber\\
    &b_1 u(0,t)+b_2 \partial_x u(0,t) = 0,\quad \partial_x u(1,t)=0
\end{align*}
We discretize the space onto $\mathbf{x}\in\mathbb{R}^{n_x-1}$ such that $x_{i+1}=i\Delta x$ for $ i=1,\dots,n_x-1$ where the grid space is given by $\Delta x=1/n_x$ to convert the F-PDE into a linear F-ODE system
\begin{align}
&^C_0 D_t^{g(\mathbf{x};\alpha)}\mathbf{u}=A\mathbf{u}\nonumber\\
&\mathbf{u}(0)=u(\mathbf{x},0)
\end{align}
where the matrix is given by $\mathbf{A}=-\mathbf{L}$, with $\mathbf{L}$ the second order centered finite difference approximation of the Laplacian
\begin{equation}
    \mathbf{L} = \frac{\mathrm{diag}(\kappa(\mathbf{x}))}{(\Delta x)^2}
    \begin{pmatrix}
    \left(\frac{b_2}{b_1 \Delta x - b_2} + 2\right) & -1 & 0 & 0 & \cdots & 0 & 0 \\
    -1 & 2 & -1 & 0 & \cdots & 0 & 0\\
    0 & -1 & 2 & -1 & \cdots & 0 & 0\\
    0 & 0 & -1 & 2 & \cdots & 0 & 0\\
    \vdots & \vdots & \vdots & \vdots & \ddots & \vdots & \vdots \\
    0 & 0 & 0 & 0 & \cdots & -1 & 1
\end{pmatrix}
\end{equation}
In a typical inverse problem set-up we would observe the solution at $u(\Delta x,T)$ and then calibrate the parameters. If the unknowns are, for example, the $\alpha$ in the vector $g(\mathbf{x};\alpha)$ and both parameters $b_1$,and $b_2$ at the Robin boundary condition, we would need to compute the following sensitivities.
\begin{equation}
    \frac{\partial u_1(T)}{\partial b_i}=\frac{\partial u_1(T)}{\partial a_{11}}\frac{\partial a_{11}}{\partial b_i}\qquad\frac{\partial u_1(T)}{\partial \alpha}=\sum_{i=1}^{n_x-1} \frac{\partial u_1(T)}{\partial g_i}\frac{\partial g_i}{\partial \alpha}.
\end{equation}
Here $g_i$ denotes $g(x_i;\alpha)$.
To show how our method works in a practical problem, we compute the quadratic loss of the solution both, with the L1 scheme $u(T;\alpha_0)$ and our stochastic method (dropping the subscript $u_1(T;\theta):=u(T;\theta)$)
\begin{equation}
    \mathcal{L}(\alpha)=\frac{1}{2}(u(T;\alpha)-u(T;\alpha_0))^2,\quad \mathcal{L}(a_{11})=\frac{1}{2}(u(T;a_{11})-u(T;(a_{11})_0))^2
\end{equation}
together with their sensitivities
\begin{align}
    \frac{\partial \mathcal{L}(\alpha)}{\partial \alpha} &=\frac{\partial u(T;\alpha)}{\partial \alpha}(u(T;\alpha)-u(T;\alpha_0))\nonumber\\ \frac{\partial \mathcal{L}(a_{11})}{\partial a_{11}}&=\frac{\partial u(T;a_{11})}{\partial a_{11}}(u(T;a_{11})-u(T;(a_{11})_0))
\end{align}
for $\alpha\in\{0.4+0.1i,\;i=0,\dots,5\}$ and $a_{11}\in\{-380.95+5i,\;i=-3,-2,\dots,1,2\}$ with true parameters $\alpha_0=0.7$ and $(a_{11})_0\approx -380.95$.

\begin{figure}[t]
    \centering
    \includegraphics[scale=0.35]{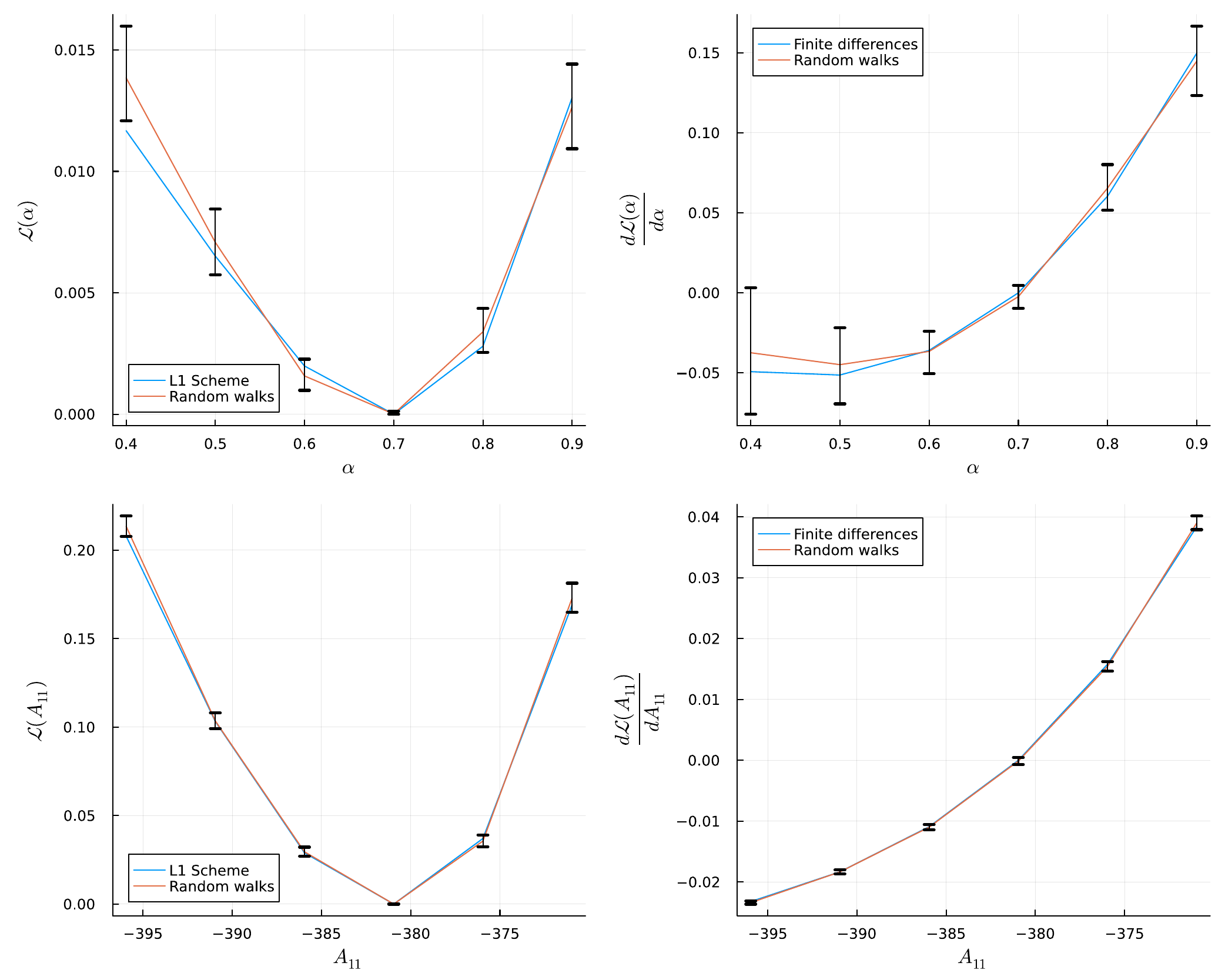}
    \caption{Deterministic and stochastic results for the loss and its sensitivities with respect to the $\alpha$ value and the first entry of the matrix $\mathbf{A}$.}
    \label{fig:all_graphs}
\end{figure}
As both, the stochastic estimator of the loss and the estimator of its sensitivity, follow generalized $\chi^2$-distributions \cite{ittrich2000probabilities}, we opted to construct $p=0.05$ confidence intervals by bootstrapping the random walks. To do so, we sample $N_s=10^6$ random walks and approximate the probability measure of the pathspace by the empirical distribution. Then, we sample with replacement---bootstrap---$N_s$ random walks from said distribution and compute the loss and sensitivity given by this bootstrapped samples. We do this $5000$ times to have a reasonable approximation of the distribution of the stochastic solution, and construct confidence intervals using the $p$ and $1-p$ quantiles of those $5000$ samples as seen in Fig. \ref{fig:all_graphs}.

The deterministic solution and the stochastic confidence intervals overlap in all except one result $\mathcal{L}(\alpha=0.4)$. As we justified in the previous section: one failure in twenty-four experiments falls within the acceptance of the null-hypothesis at $p=0.05$.

\section{Conclusion}
\label{conclusions}
In this work, we have derived a stochastic solution based on random walks for a large class of linear F-ODE systems. A notable feature of this solution is that it can be used to compute the solution at single nodes. A further advantage of this approach lies in the capability of being able to compute all unbiased sensitivities of the solution through discrete event system differentiation,  while avoiding time discretization and intermediate storage of the solution.

The expected execution time of the algorithm and its variance are theoretically estimated, and numerical results are provided to hold our claims. In view of the results,  we consider that this algorithm stands as a viable alternative to the deterministic methods for solving large-scale linear F-ODE systems. However, to apply this method successfully it may be required to use some variance reduction tools. In particular, conditional Monte-Carlo is well-suited for the semi-markovian nature of our problem~\cite{anderson2022conditional}. 


One advantage of our approach compared to stochastic calculus methods for F-PDEs is that users do not need to deal with first exit time issues after discretization. Additionally, our method works with arbitrary matrices, whereas traditional stochastic calculus methods are often restricted to a more limited class of linear operators in space. Furthermore, we do not discretize in time, meaning our only source of bias is the initial spatial discretization.

It is left for future research to generalize the algorithm to account for a forcing term and continuum observations. In the continuous case, our algorithm still holds by providing a piecewise constant estimator for the solution in $[0,T]$, however, some of our sensitivity estimators are not piecewise constant. In the particular application to inverse problems, an efficient implementation may require taking into account the specific optimization algorithm used to solve the problem. Concerning forcing terms, this can be indeed added to the recursive integral representation of the solution and provided readily a stochastic interpretation. However, to simplify our exposition we decided not to include here and left as well for a future work.

To conclude, it is worth pointing out here that the proposed method may arise as a competitive method for solving linear F-ODE based optimization problems in a massive parameter space. Although theoretically this seems to require taking the limit $N_s \rightarrow \infty$ to obtain the exact solution and gradient, approximations of both quantities are often sufficient. For instance, stochastic gradient descent has achieved enormous success in deep learning over the last decade, sometimes even proving advantageous over standard gradient descent \cite{hardt2016train}.


\bmhead{Competing interests}

All authors certify that they have no affiliations with or involvement in any organization or entity with any financial interest or non-financial interest in the subject matter or materials discussed in this manuscript.

\bmhead{Data availability}

The authors declare that the data sets generated during the current study are available from the corresponding author on reasonable request.

\begin{appendices}
\section{Proofs}
\subsection{Derivation of the time sensitivity formula}
\label{appendix-A}
For any two random variables $X,Y$ defined in the same probability space $(\Omega,\mathcal{F},\mathbb{P})$ we have \cite{heidergott2008measure}
\begin{equation}
    \mathbb{E}(X)=\mathbb{E}(\mathbb{E}(X\vert Y))
\end{equation}
Our idea will be to condition with some random variable $Y$ so we can use dominated convergence arguments (Chapters 1 and 3 in \cite{fu_conditional_1997}) to swap the derivative and the expected value
\begin{equation}
    \partial_T \mathbb{E}(X)=\mathbb{E}(\partial_T \mathbb{E}(X\vert Y))
\end{equation}
while having some closed form for $\partial_T \mathbb{E}(X\vert Y)$. We firstly condition over the $\sigma$-algebra generated by $\Gamma_s$ for $0< s \leq T$, which represents the information we obtain from the hypothetical whole path $\Gamma_t\colon [0,\infty)\rightarrow \bm{S}$ when we simulating it up to time $T$.
\begin{align}
    &\mathbb{E}(J_{T+\Delta T}(\Gamma_t)-J_{T}(\Gamma_t)\mid \Gamma_0 = i)\nonumber\\
    &=\mathbb{E}(\mathbb{E}(J_{T+\Delta T}(\Gamma_t)-J_{T}(\Gamma_t)\mid \sigma(\Gamma_s\colon 0< s \leq T))\mid \Gamma_0 = i)
\end{align}
We need to condition one last time over the number of jumps made between $T$ and $T+\Delta T$ conditioned over the whole story of the path $\Gamma_t\colon 0< s \leq T$
\begin{align}
&\mathbb{E}(J_{T+\Delta T}(\Gamma_t)-J_{T}(\Gamma_t)\mid \sigma(\Gamma_s\colon 0< s \leq T))\nonumber\\
&=\mathbb{E}(\mathbb{E}(J_{T+\Delta T}(\Gamma_t)-J_{T}(\Gamma_t)\mid \nu_{T+\Delta T}-\nu_T) \mid \sigma(\Gamma_s\colon 0< s \leq T))
\end{align}
The main idea in \cite{fu_conditional_1997} is to perform the following swapping
\begin{equation}
    \partial_T \mathbb{E}(X\mid \Gamma_0=i)=\mathbb{E}(\mathbb{E}(\partial_T \mathbb{E}(X\mid \nu_{T+\Delta T}-\nu_T)\mid \sigma(\Gamma_s\colon 0< s \leq T))\mid\Gamma_0=i)
\end{equation}
and compute both outer expectations via Monte-Carlo while using a closed form for the derivative of the innermost expectation conditioned to the Monte-Carlo values. As $\sigma(\Gamma_0,\Gamma_s\colon 0< s \leq T)=\sigma(\Gamma_s\colon 0\leq s \leq T)$, we use a superscript $^{\Gamma_T}$ to indicate conditioning over the $\sigma$-algebra generated up to time $T$.

\begin{align}
    &\lim_{\Delta T \shortarrow{7}  0}\frac{1}{\Delta T}\mathbb{E}^{\Gamma_T}(J_{T+\Delta T}(\Gamma_t)-J_{T}(\Gamma_t)\mid \nu_{T+\Delta T}-\nu_T)=\nonumber\\
    &\sum_{k=0}^\infty \lim_{\Delta T \shortarrow{7}  0}\frac{\mathbb{P}^{\Gamma_T}(\nu_{T+\Delta T}-\nu_T=k)}{\Delta T}\mathbb{E}^{\Gamma_T}(J_{T+\Delta T}(\Gamma_t)-J_{T}(\Gamma_t)\mid \nu_{T+\Delta T}-\nu_T=k)
\end{align}
Using the arguments in~\cite{fu_conditional_1997}, where it is shown that the limit of jumping $k$ times in a $\Delta T$ stretch of time is of order
\begin{equation}
    \mathbb{P}(\nu_{T+\Delta T}-\nu_T=k\mid \sigma(\Gamma_s\colon 0\leq s \leq T))=\mathcal{O}(\Delta T^k)
\end{equation}
we drop all events $k\geq 2$. We also drop the event of no jump $k=0$ as the functional does not change if there are no extra jumps, that is $J_{T+\Delta T}(\Gamma_t)=J_{T}(\Gamma_t)$ if $\nu_{T+\Delta T}=\nu_T$. For the event $k=1$, that is, just one jump in the $\Delta T$ timespan, we want to compute the following limit\footnote{This is referred to as stochastic weight in \cite{arya2022automatic} and critical rate in \cite{fu_conditional_1997}.}
\begin{equation}
\label{critical-rate-primordial}
    \lim_{\Delta T \shortarrow{7} 0}\frac{1}{\Delta T}\mathbb{P}(\nu_{T+\Delta T}-\nu_T=1\mid \nu_T,\{S_i,\Delta \tau_i \}_{i=1}^{\nu_T}, S_{1+\nu_T})
\end{equation}
where we explicitly write the random variables that generate the $\sigma$-algebra for which we have an actual value in each simulation.
Notice how we do not condition over the exact value of the last jump $\Delta \tau_{1+\nu_T}$ as the only information we get from observing the path up to time $T$ is that $\tau_{1+\nu_T}=T-\sum_{i=1}^{\nu_T}\Delta \tau_i$. To finish up, we easily find a closed-form for Equation (\ref{critical-rate-primordial})
\begin{align}
\label{eq:critical-rate}
&\lim_{\Delta T \shortarrow{7} 0}\frac{1}{\Delta T}\frac{\mathbb{P}(T+\Delta T-\sum_{i=1}^{\nu_T}\Delta \tau_i >\xi(S_{1+\nu_T})>T-\sum_{i=1}^{\nu_T}\Delta \tau_i)}{\mathbb{P}(\xi(S_{1+\nu_T})>\sum_{i=1}^{\nu_T} \Delta \tau_i)}\nonumber\\
&=\frac{\mathbb{P}(\xi(S_{1+\nu_T})=T-\sum_{i=1}^{\nu_T} \Delta \tau_i)}{\mathbb{P}(\xi(S_{1+\nu_T})>T-\sum_{i=1}^{\nu_T}\Delta \tau_i}
\end{align}
In the numerator we have the probability of the path making an extra jump if and only if $\Delta T$ is added as extra time, given that the path had only jumped $\nu_T$ times and we were in state $S_{1+\nu_T}$ at time $T$, so we know the sojourn time $\xi_{1+\nu_T}$. The limit of the numerator is, of course, the density of the sojourn time. Lastly
\begin{align*}
     &\mathbb{E}(J_{T+\Delta T}(\Gamma_t)-J_{T}(\Gamma_t)\mid \nu_{T+\Delta T}-\nu_T=1,\nu_T,\{S_i,\Delta \tau_i \}_{i=1}^{\nu_T}, S_{1+\nu_T})\nonumber\\&=J_T(\Gamma)\left(\frac{\mathrm{sgn}(a_{S_{1+\nu(T)},S_{2+\nu(T)}})\sum_{j\neq S_{1+\nu(T)}}\mathrm{sgn}(a_{S_{1+\nu(T)}j})a_{S_{1+\nu(T)}j}}{-a_{S_{1+\nu(T)},S_{1+\nu(T)}}}-1\right)\nonumber\\
    &:=J_{1+\nu(T)}(\Gamma)-J_{\nu(T)}(\Gamma)
\end{align*}
notice how we can abuse the notation from $J_T$ to $J_{\nu(T)}$, or in general, to $J_N, N\in \mathbb{Z}_{\geq 0}$ as the functional $J$ only depends on the number of jumps made, not the sojourn times themselves. Therefore our unbiased CMC estimator is given by
\begin{align}
    \partial_T u_i(T)=\mathbb{E}\left(\frac{\mathbb{P}(\xi(S_{1+\nu_T})=T-\sum_{i=1}^{\nu_T} \Delta \tau_i)}{\mathbb{P}(\xi(S_{1+\nu_T})>T-\sum_{i=1}^{\nu_T} \Delta \tau_i)}(J_{1+\nu(T)}(\Gamma)-J_{\nu(T)}(\Gamma))\right)
\end{align}
Our results exactly coincide with the result seen in Chapter 5.1. in \cite{fu_conditional_1997}, where the critical rate for a critical jump is also seen to be the hazard function.
\begin{remark}
As it can be seen in these formulas or in \cite{fu_conditional_1997} Chapter 3, closed forms for the critical rate are difficult to obtain. If we were to apply these same ideas to the vector $\bm{\alpha}$, we would have to find a way of computing fast densities of fractional phase-type distributions \cite{albrecher2020multivariate}, which is unadviced. CMC ideas applied to entries of $a_{jk}$ become so difficult that naturally lead to stochastic automatic differentiation ideas and importance sampling \cite{arya2022automatic}.    
\end{remark}
\subsection{Density function in the pathspace}
\label{appendix-B}
Instead of studying $\Gamma_t:[0,T]\times \Omega \rightarrow \bm{S}$ as a stochastic process over the random number stream of the computer $\Omega$, we can split the set $\{\Gamma_t:[0,T]\times \Omega \rightarrow \bm{S}\}$, as a probability space, as follows
\begin{equation}
    \bigcup_{k=0}^\infty\{\nu_T=k\}=\bigcup_{\nu_T=0}^\infty \left\{\{S_j,\Delta \tau_j\}_{j=1}^{1+\nu_T}\mid \sum_{j=1}^{\nu_T} \Delta \tau_j < T\leq \sum_{j=1}^{1+\nu_T} \Delta \tau_j \right\}
\end{equation}
where the union is disjoint. This is the consequence of a path being fully described by $\nu_T$ and its sequence of $1+\nu_T$ states and sojourn times. By splitting the probability space, the sojourn times are found to be in an adequate subset of $\mathbb{R}^{1+\nu_T}$ for each fixed $\nu_T$, which allows us to make use of density functions once we fix a value for $\nu_T$.

To obtain the Malliavin weights formula for a fixed spatial path, we condition over $\nu_T$, which recasts the expected value as a weighted sum of conditional expected values.
\begin{align}
    &\partial_\theta\mathbb{E}_\theta(J_\theta(\Gamma_t))=\partial_\theta\mathbb{E}_\theta(\mathbb{E}_\theta(J_\theta(\Gamma_t)\mid \nu_T))\nonumber\\
    &=\partial_\theta\sum_{k=0}^\infty \mathbb{P}_\theta(\nu_T=k)\mathbb{E}_\theta(J_\theta(\Gamma_t)\mid \nu_T=k)\nonumber\\&=\sum_{k=0}^\infty \mathbb{P}_\theta(\nu_T=k)(\partial_\theta\log(\mathbb{P}_\theta(\nu_T=k))\mathbb{E}_\theta(J_\theta(\Gamma_t)\mid \nu_T=k)+\partial_\theta\mathbb{E}_\theta(J_\theta(\Gamma_t)\mid \nu_T=k))
\end{align}
The sum over $k$ is in practice taken via Monte-Carlo simulation. The conditional expected values are integrals over Euclidean subsets weighted by the density given in Equation \ref{p_gamma} \cite{ertel2022operationally}, so we can rigorously apply the Malliavin weight technique \cite{warren2013malliavin}, which yields the same formulas found in Section \ref{sec:mall}. Notice that the left term $(\partial_\theta\log(\mathbb{P}_\theta(\nu_T=k))\mathbb{E}_\theta(J_\theta(\Gamma_t)\mid \nu_T=k)$ is cancelled by the right term, as $\mathbb{P}_\theta(\nu_T=k)$ appears as a denominator normalizing the Euclidean subsets.
This argument also applies to spatial parameters.

\end{appendices}
\bibliography{sn-bibliography}

\end{document}